\documentclass[11pt,reqno]{amsart}
\usepackage[utf8]{inputenc}  
\usepackage[T1]{fontenc}
\usepackage{graphicx}
\usepackage{cite}
\usepackage{amssymb, amsmath,stmaryrd,cancel}
\usepackage[normalem]{ulem}

\usepackage{fixltx2e}
\usepackage{newunicodechar}
\usepackage{color,xcolor}

\usepackage{hyperref}

\newtheorem{thm}{Theorem}[section]
\newtheorem{lem}[thm]{Lemma}
\newtheorem{cor}[thm]{Corollary}
\newtheorem{prop}[thm]{Proposition}
\newtheorem{rem}[thm]{Remark}

\newtheorem{rems}[thm]{Remarks}

\theoremstyle{definition}

\theoremstyle{remark}

\numberwithin{equation}{section}

\newcommand{\N}{\mathbb N}
\newcommand{\R}{\mathbb R}

\newcommand{\kL}{\mathcal L}
\newcommand{\ve}{\varepsilon}

\newcommand{\rd}{\mathrm{d}}
\newcommand{\p}{\partial}

\newcommand{\ww}{\mathrm{w}}
\newcommand{\aw}{\mathrm{w}_0}

\allowdisplaybreaks

\begin{document}

\title[Recovering Initial States in Quasilinear Parabolic Problems]{Recovering Initial States in Certain Quasilinear Parabolic Problems from Time Averages}

 \author{Bogdan--Vasile Matioc}
\address{Fakult\"at f\"ur Mathematik, Universit\"at Regensburg,   93053 Regensburg, Germany.}
\email{bogdan.matioc@ur.de}

\author{Christoph Walker}
\address{Leibniz Universit\"at Hannover,
Institut f\"ur Angewandte Mathematik,
Welfengarten 1,
30167 Hannover,
Germany.}
\email{walker@ifam.uni-hannover.de}

\begin{abstract}
The inverse problem of reconstructing the initial state in quasilinear parabolic equations from time averages is investigated. Under suitable regularity assumptions on the quasilinear structure and a superlinear growth condition near zero for the semilinear part, it is shown that the initial state can be uniquely recovered from small time averages taken over an arbitrary time period. The applicability of the result is demonstrated for certain chemotaxis models and reaction–diffusion systems.
\end{abstract}


\maketitle

\section{Introduction and Main Result}

Recovering the past state of a system from given observations is a central challenge in the study of evolution equations. In this work,  we focus our attention on quasilinear parabolic problems of the form
\begin{subequations}\label{EE}
\begin{equation}\label{EEeq}
u'=A(u)u+ f(u)\,,\quad t\in(0,T]\,,
\end{equation}
subject to 
\begin{equation}\label{EEini}
\int_0^T  \ww(t)\, u(t)\,\rd t=M\,.
\end{equation}
\end{subequations}
We assume that $E_1$ and $E_0$ are Banach spaces such that $E_1$ is continuously and densely embedded into $E_0$ and that
 $A(u)\in\mathcal{H}(E_1,E_0)$. That is, for fixed $u$, the operator $A(u)$ with domain $E_1$ generates  a strongly continuous analytic semigroup on $E_0$. Hereby, $A(u)$ may depend nonlinearly upon $u$. Moreover, $f= f(u)$ is a  nonlinear function,  $\ww$ is a nontrivial weight function, and $M\in E_1$. In~\eqref{EE}, the usual initial condition is replaced by a nonlocal condition~\eqref{EEini}.
Roughly speaking, we shall show that if $A=A(u)$ and $f=f(u)$ depend Lipschitz continuously upon $u$ and $f$ has a superlinear growth near zero, then one can recover the (unique) initial state $u(0)$ from small values of time averages over arbitrary time
periods. Hence, problem~\eqref{EE} is well-posed for arbitrary $T>0$ and small values of $M$. The smallness condition is due
to the fact that $T >0$ is a priori fixed and one thus seeks for global solutions to a nonlinear problem. We shall later also consider related problems, where the nonlocal condition~\eqref{EEini} is replaced by conditions of the form
\begin{equation}\label{other1}
u(0)+\int_0^T  \ww(t)\, u(t)\,\rd t=M
\end{equation}
or (with a constant $\aw$)
\begin{equation}\label{other2}
 u(0)- \aw u(T)=M\,,
\end{equation}
which, however, are significantly easier to handle.\\

Before presenting our results in detail, let us mention that the semilinear counterpart of \eqref{EE} (i.e., when $A(u)$ is independent of $u$) was recently investigated in \cite{SchmitzWalkerJDE}. As we shall see herein, however, the quasilinear setting turns out to be substantially more involved. 

The linear counterpart of problem~\eqref{EE} (also subject to the other nonlocal conditions~${\text{\eqref{other1}-\eqref{other2}}}$) was considered in \cite{Dokuchaev19} in a Hilbert space setting.
Further results on solvability, uniqueness, and regularity of solutions to  linear  parabolic problems with conditions of type \eqref{other1}–\eqref{other2}  were obtained  in~\cite{Deng93,Dokuchaev4,Dokuchaev6,Dokuchaev7,MVS19,Pao95} (see also the references therein). The connection to stochastic differential equations is explored in~\cite{Dokuchaev4,Dokuchaev6,Dokuchaev8}.

Moreover, both linear and nonlinear problems with alternative, and in some cases more general, nonlocal conditions have been studied in a number of works~\cite{BenedettiEtal17,BenedettiCiani22,GarciaFalset08,Hernandez18,Jabeen17,Jackson93,Vrabie18, TrietEtal_MMAS21}, employing a wide range of analytical techniques in concrete functional settings~\cite{BenedettiCiani22,Jackson93,MVS19, TrietEtal_MMAS21} as well as in abstract frameworks~\cite{BenedettiEtal17,GarciaFalset08,Hernandez18,Jabeen17,Vrabie18}; see also the survey~\cite{Ntouyas05}. These references are by no means exhaustive. It is worth noting, however, that only a few contributions explicitly deal with genuinely quasilinear structures (but see~\cite{GarciaFalset08, TrietEtal_MMAS21}). \\

In order to  explain the difficulties associated with~\eqref{EE}, we proceed formally and assume the existence of a (H\"older continuous in time) solution $u$ such that the mapping $t \mapsto A(u(t))$ generates a parabolic evolution operator $U_{A(u)}(t,s)$, $0 \leq s \leq t \leq T$, on $E_0$ with regularity subspace $E_1$ in the sense of~\cite{LQPP}. Then, $u$ solving~\eqref{EEeq} can be written in form of the variation-of-constants formula
$$
u(t)=U_{A(u)}(t,0)u(0)+\int_0^t U_{A(u)}(t,s) f(u(s))\,\rd s\,,\quad t\in [0,T]\,,
$$
and, when plugging this into~\eqref{EEini}, we obtain an equivalent reformulation of~\eqref{EE} as a fixed point problem
\begin{subequations}\label{vdk}
\begin{equation}\label{vdke}
u(t)=U_{A(u)}(t,0)\Xi(u)+\int_0^t U_{A(u)}(t,s) f(u(s))\,\rd s\,,\quad t\in [0,T]\,,
\end{equation}
where $u(0)=\Xi(u)$ is given by
\begin{equation}\label{ulin}
\Xi(u)=\Phi_{A(u)}^{-1}\big(M-\Psi(u) f(u)\big)
\end{equation}
\end{subequations}
with
\begin{equation}\label{Phi}
\Phi_{A(u)}:=\int_0^T\ww(t)\,U_{A(u)}(t,0)\,\rd t
\end{equation}
and 
\begin{equation}\label{Psi}
\Psi(u)g:= \int_0^T\ww(t)\int_0^t U_{A(u)}(t,s) g(s)\,\rd s\,\rd t\,.
\end{equation}
Clearly, solving~ the fixed point problem~\eqref{vdk} requires on the one hand the existence of the evolution operator~$U_{A(u)}$ and, on the other hand, a careful analysis of the operator~$\Phi_{A(u)}$ that encodes the initial condition in~\eqref{vdk}, in particular concerning its invertibility and its dependence on~$u$.
It turns out that, under rather mild assumptions,~$\Phi_{A(u)}$ is an isomorphism between~$E_0$ and~$E_1$,  so that the  inverse operator $\Phi_{A(u)}^{-1}\in\mathcal{L}(E_1,E_0)$ is of  the same order as $A(u)$.
Hence, one can only expect that the initial value $u(0)=\Xi(u)$ in~\eqref{vdk} lies in~$E_0$, without any additional a priori regularity. It is well-known, however, that an initial value in the phase space $E_0$ is usually insufficient to guarantee the well-posedness of quasilinear problems, e.g. see~\cite{Amann_Teubner,Am88,PS16}. Furthermore, it is necessary that $M\in E_1$.
In view of these restrictions, and in order to retain greater flexibility for later applications, it is natural to reformulate the problem in spaces of higher regularity within an interpolation–extrapolation scale (see below). Moreover, the use of time-weighted spaces enables us to handle the singular behavior of the evolution operator ~$U_{A(u)}(t,s)$ in this scale at $t=s$ and to accommodate a sufficiently broad class of semilinearities. A further crucial step in solving~ the fixed point problem~\eqref{vdk} is then establishing the (Lipschitz) continuity of the mapping $u \mapsto \Phi_{A(u)}^{-1}$ in the operator norm of $\mathcal{L}(E_1,E_0)$,  which will require some effort.

\subsection*{Main Result}

To be more precise, certain technical assumptions are necessary. For this purpose, we assume for the (real or complex) Banach spaces $E_1$ and $E_0$ that
\begin{subequations}\label{EQ:A}
\begin{equation}\label{a2}
\text{$E_1$ is compactly and densely embedded in $E_0$}\,.
\end{equation}
Further, let
\begin{equation}\label{a3}
A(0)\in\mathcal{H}(E_1,E_0)\,,
\end{equation}
i.e. $A(0)\in\mathcal{L}(E_1,E_0)$ is the generator of a strongly continuous analytic semigroup~$(e^{tA(0)})_{t\ge 0}$ on~$E_0$. Define then
$$
E_2:=\big(\mathrm{dom}(A(0)^2),\|(\omega-A(0))^2\cdot\|_{E_0}\big)
$$ 
for some fixed $\omega\in \rho(A(0))$ (i.e. in the resolvent set).
We fix for each $\theta\in (0,1)$ an exact admissible interpolation functor~$(\cdot,\cdot)_\theta$  of exponent $\theta$ and introduce the Banach spaces\footnote{That is, $[(E_\theta,\tilde A_\theta);0\le \theta\le 2]$ is (part of) \textit{the interpolation-extrapolation scale of order $0$ generated by $(E_0,\tilde A)$ and~$(\cdot,\cdot)_\theta$} in the sense of \cite[V.Theorem~1.5.1]{LQPP}, where $\tilde A:=\omega-A(0)$. Then $E_\theta\hookrightarrow E_\vartheta$  is a compact and dense embedding for $0\le \vartheta\le \theta\le 2$.}
\begin{equation*}
E_\theta:=(E_0,E_1)_\theta\,,\qquad E_{1+\theta}:=(E_1,E_2)_\theta\,.
\end{equation*}
Fix
\begin{equation}\label{a3b}
0\le \beta<\alpha<\alpha_0\le 1
\end{equation}
 and    assume that
there is $r_0>0$ such that
\begin{equation}\label{a5}
\begin{aligned}
A&\in C^{1-}\big(\mathbb{B}_{E_\beta}(0,r_0), \mathcal{L}(E_1,E_0)\big)\cap C^{2-}\big(\mathbb{B}_{E_\beta}(0,r_0), \mathcal{L}(E_{1+\alpha},E_\alpha)\big)\\
&\quad\cap C^{1-}\big(\mathbb{B}_{E_\beta}(0,r_0), \mathcal{L}(E_{1+\alpha_0},E_{\alpha_0})\big)\,,
\end{aligned}
\end{equation}
where $C^{1-}$ and $C^{2-}$ refer to (local) Lipschitz continuity of a function and, respectively, its derivative.
We also require for the scalar weight function that
\begin{equation}\label{a1}
\ww\in C^1([0,T])\,,\qquad \ww(0)\not=0\,.
\end{equation}
As we shall later prove in Proposition~\ref{P1}, assumptions~\eqref{a3} and~\eqref{a1} imply that
$$
\Phi_{A(0)}=\int_0^T\ww(t)e^{tA(0)}\,\rd t\in\mathcal{L}(E_0,E_1)\,.
$$ 
We then further suppose that
\begin{equation}\label{a6}
\mathrm{ker}(\Phi_{A(0)})=\{0\}
\end{equation}
\end{subequations}
and thus obtain (again from Proposition~\ref{P1} below) that $\Phi_{A(0)}\in\mathcal{L}is(E_0,E_1)$. 

As for the semilinear part we fix
\begin{subequations}\label{AA}
\begin{equation}\label{A1Y}
 \ell>0\,,\qquad \alpha_0<\gamma\le 1\,,\qquad 0<\xi<\min\left\{\alpha+\frac{1}{\ell+1},1\right\}\,,
\end{equation}
and assume that $f:E_\xi\cap \mathbb{B}_{E_\beta}(0,r_0)\to E_\gamma$ 
satisfies, for $v,w\in E_\xi\cap \mathbb{B}_{E_\beta}(0,r_0)$,
\begin{equation}\label{F2}
\|f(v)-f(w)\|_{E_\gamma}\le c(r_0) \left(\left[\|v\|_{E_\xi}^\ell+\|w\|_{E_\xi}^\ell\right] \|v-w\|_{E_\xi}+ \left[\|v\|_{E_\xi}^{\ell+1}+\|w\|_{E_\xi}^{\ell+1}\right] \|v-w\|_{E_\beta}\right)
\end{equation}
for some constant $c(r_0)>0$ and that
\begin{equation}\label{F2x}
\qquad f(0)=0\,.
\end{equation}
\end{subequations}
We can then state the main result regarding the solvability of~\eqref{EE}:

\begin{thm}\label{MT1}
Let $T>0$ and assume~\eqref{EQ:A} and \eqref{AA}.  Then, there exists $m_0>0$ such that \eqref{EE} has for each $M\in E_{1+\alpha}$ with $\|M\|_{E_{1+\alpha}}\le m_0$  a unique solution
$$
u\in C^{\min\{\alpha-\beta, 1-\mu(\ell+1)\}}\big([0,T],E_\beta\big)\cap  C\big([0,T],E_\alpha\big)\cap C\big((0,T],E_1\big)\cap C^{1}\big((0,T],E_0\big)
$$
such that $\lim_{t\to 0^+}t^\mu\|u(t)\|_{E_\xi}=0$ for\footnote{Given $x\in\R$, we set $x_+:=\max\{0,x\}$.}  $(\xi-\alpha)_+<\mu< (\ell+1)^{-1}$.
\end{thm}

That is, problem~\eqref{EE} is well-posed over an arbitrary time interval $(0,T)$ provided that $M$ has a sufficiently small norm in $E_{1+\alpha}$. Let us emphasize again that the smallness condition on $M$ (as well as the condition $f(0)=0$)  is required for establishing a global existence result for a nonlinear problem. In fact, it suffices that $M$ is sufficiently close to the average of an equilibrium:

\begin{rem}
If $v_*\in E_{1+\alpha}$ is an equilibrium, i.e. $A(v_*)v_*+f(v_*)=0$, then one may consider the equation satisfied by $w:=u-v_*$: 
\begin{equation}\label{*}
w'=A_*(w)w+ f_*(w)\,,\quad t\in(0,T]\,,\qquad \int_0^T  \ww(t)\, w(t)\,\rd t=M-N_*\,,
\end{equation}
with
\begin{align*}
A_*(w):=A(w+v_*)\,,\qquad f_*(w):=A(w+v_*)v_*+f(w+v_*)\,,
\end{align*}
and
$$
N_*:=\int_0^T  \ww(t)\, \rd t\, v_*\in E_{1+\alpha}\,.
$$
Under corresponding assumptions, the well-posedness of \eqref{*} for $\|M-N_*\|_{E_{1+\alpha}}$ sufficiently small can then be addressed as for \eqref{EE}. Note that Theorem~\ref{MT1} is stated for the trivial equilibrium $v_*=0$ (recall that~$f(0)=0$).
\end{rem}

Some remarks are in order regarding the assumptions imposed in Theorem~\ref{MT1}:

\begin{rems}\,
{\bf (a)} Assumption~\eqref{a5} on the local Lipschitz continuity of $A$ in the extrapolated spaces~$E_{1+\alpha}$ and~$E_{1+\alpha_0}$ is imposed to allow for a more flexible functional-analytic framework in applications.  
The condition $\alpha>\beta$ in~\eqref{a3b} guarantees the a priori Hölder continuity of solutions (and hence of $t\mapsto A(u(t))$), which is essential for defining the evolution operators.  
For uniformly elliptic differential operators, assumptions~\eqref{a3b}-\eqref{a5} essentially amount to regularity requirements on the coefficients (see Section~\ref{Sec5}).

{\bf (b)} In applications, assumption~\eqref{a6} can most conveniently be verified in a Hilbert space setting by means of Fourier series; see \cite[Section~6]{SchmitzWalkerJDE} (and also~\cite{Dokuchaev19}).

{\bf (c)} Assumption~\eqref{AA} is often met in applications, such as reaction-diffusion systems, and compatible within the framework of time-weighted spaces (that we shall use to solve~\eqref{vdk}). See, for instance~\cite[Lemma~4.1]{MW_PRSE}.

{\bf (d)} The assumption $\ww(0)\neq 0$ in~\eqref{a1} ensures that some information about $u(0)$ is encoded in~\eqref{EEini}.

\end{rems}

The proof of Theorem~\ref{MT1} relies on solving the fixed point problem~\eqref{vdk} via Banach's fixed point theorem. To this end, some preparatory steps are required, the key one being the analysis of the operator~$\Phi_{A(u)}$ introduced in~\eqref{Phi} that encodes the initial condition.  

In a first step, we investigate in Section~\ref{Sec2} the operator~$\Phi_A$ for a given $A\in C^\rho([0,T],\mathcal{H}(E_1,E_0))$. We show that $\Phi_A \in \mathcal{L}(E_0,E_1)$ is a Fredholm operator (under the compactness assumption~\eqref{a2}) and hence invertible if injective (see Proposition~\ref{P1}). Moreover, we  prove that the mapping 
\[  
C^\rho\big([0,T],\mathcal{H}(E_1,E_0)\big) \rightarrow \mathcal{L}(E_0,E_1)\,,\quad A \mapsto \Phi_A
\] 
is locally Lipschitz continuous (see Theorem~\ref{T1}). This result is derived via a careful analysis of the construction of the evolution operators~$U_A$ in~\cite{LQPP}. It is important to emphasize that continuity is established in the topology of $\mathcal{L}(E_0,E_1)$, rather than in proper interpolation spaces (see Remark~\ref{R5} for more details). In fact, using the interpolation–extrapolation scale, all these results remain valid if $E_0$ and $E_1$ are replaced by the spaces $E_\alpha$ and $E_{1+\alpha}$, respectively.  

In Section~\ref{Sec3}, we then first derive stability estimates for the evolution operators~$U_{A(u)}$ and the operator~$\Phi_{A(u)}^{-1}$ with respect to a given $u$ in the interpolation–extrapolation scale. With these preparatory steps we are finally in a position to prove Theorem~\ref{MT1} by a fixed point argument at the end of Section~\ref{Sec3}.

In Section~\ref{Sec5} we shall provide applications of Theorem~\ref{MT1} to recover the initial states in (a variant of) the classical chemotaxis model of Keller-Segel~\cite{KS1971} as well as in quasilinear reaction-diffusion models with coefficients depending nonlocally on the solution.

Finally, Section~\ref{Sec4} is dedicated to the other nonlocal conditions~\eqref{other1} and ~\eqref{other2}. We prove in Theorem~\ref{MT3} and Theorem~\ref{MT4} similar results on well-posedness of the corresponding quasilinear evolution problem as in Theorem~\ref{MT1}.

\section{The Operator $\Phi$: Linear Theory}\label{Sec2}

To begin with the linear theory, we assume that $E_1$ and $E_0$  are   Banach spaces  with $E_1$  densely embedded  in $ E_0$ and consider 
\begin{equation}\label{Gen}
A\in C^\rho\big([0,T],\mathcal{H}(E_1,E_0)\big)
\end{equation} 
for some $\rho\in(0,1)$ and
\begin{equation}\label{b}
\ww\in C^1([0,T])\,.
\end{equation} 
Recall from \cite[II.Corollary~4.4.2]{LQPP} that \eqref{Gen} implies that there is a unique evolution operator 
$$
U_{A}(t,s)\,,\quad 0\le s\le t\le T\,,
$$ 
for $A$ on $E_0$ with regularity subspace $E_1$ in the sense of~\cite[II.Section~2.1]{LQPP}. We shall first show
 that, under suitable additional assumptions, the operator $\Phi_{A}$ 
  defined by
 \begin{equation}\label{PhiL}
\Phi_A:=\int_0^T\ww(t)\,U_{A}(t,0)\,\rd t
\end{equation}
 is an isomorphism from $E_0$ to $E_1$ and that it depends Lipschitz continuously on~$A$. For this purpose, 
 we follow closely the construction of evolution operators in~\cite{LQPP}.
 
 \begin{rem}\label{R12}
For time-independent $A\in\mathcal{H}(E_1,E_0)$ (i.e. $A$ is the generator of an analytic semigroup~$(e^{t A})_{t\ge 0}$ so that $U_A(t,0)=e^{tA}$), the  invertibility of 
$$
\Phi_{A}=\int_0^T \ww(t) e^{t A}\,\rd t \in\mathcal{L}(E_0,E_1)
$$
has been investigated in \cite[Proposition 4.1]{SchmitzWalkerJDE}. 
\end{rem}

\subsection{Invertibility of $\Phi_A$}

We  study the invertibility of $\Phi_A$. The main result in this regard is the following:

\begin{prop}\label{P1} 
If \eqref{Gen} and \eqref{b} are satisfied, then\footnote{The property  $\Phi_A\in \mathcal{L}(E_0,E_1)$ still holds under the weaker assumption  $\ww\in C([0,T])\cap C^1([0,\ve])$, for some $\ve>0$, in place of \eqref{b}.} $\Phi_A\in \mathcal{L}(E_0,E_1)$. Moreover, if 
\begin{equation}\label{comp}
\text{$E_1$ is compactly embedded into $E_0$}
\end{equation}
and 
\begin{equation}\label{b0}
\ww(0)\not= 0\,,
\end{equation}
then $\Phi_A\in \mathcal{L}(E_0,E_1)$ is a Fredholm operator of index zero. In particular, if $\mathrm{ker}(\Phi_A)=\{0\}$ in addition, then $\Phi_A\in \mathcal{L}is(E_0,E_1)$.
\end{prop}

\begin{proof}
{\bf (i)} Assume \eqref{Gen} and \eqref{b}. Recall from \cite[II.Section~4.3]{LQPP} that the evolution operator~$U_{A}$ is given in the form
\begin{equation}\label{UA}
U_{A}(t,s)=(a_{A}+a_A*w_A)(t,s)=a_A(t,s)+\int_s^t a_A(t,\tau)w_A(\tau,s)\,\rd \tau \,,\quad 0\le s{  <} t\le T\,,
\end{equation}
with
\begin{equation}\label{aA}
a_A(t,s):=e^{(t-s)A(s)}\,,\quad 0\le s< t\le T\,,
\end{equation}
and
\begin{equation}\label{wA}
w_A:=\sum_{n=1}^\infty \underbrace{k_A*\ldots *k_A}_{n\text{ times}}\qquad\text{with}\qquad k_A(t,s):=[A(t)-A(s)]a_A(t,s)\,,\quad 0\le s  < t\le T\,.
\end{equation}
Therefore, recalling~\eqref{PhiL},
\begin{equation}\label{desc1}
\Phi_A= \Phi_{A,1}+\Phi_{A,2}\,,
\end{equation}
where
$$
\Phi_{A,1}:=\int_0^T\ww(t)\, e^{tA(0)}\,\rd t\,,\qquad \Phi_{A,2}:=\int_0^T \ww(t)\int_0^t  e^{(t-s)A(s)}w_A(s,0)\,\rd s \rd t\,.
$$
 We  note that
\begin{align}
A(0)\Phi_{A,1}&=\int_0^T\ww(t)\,A(0)\,e^{tA(0)}\,\rd t=-\ww(0)+\ww(T)\,e^{TA(0)}-\int_0^T{\ww'(t)}\,e^{tA(0)}\, \rd t\,,\label{p1}
\end{align}
hence $A(0)\Phi_{A,1}\in \mathcal{L}(E_0)$. Consequently, we have shown that $\Phi_{A,1}\in \mathcal{L}(E_0,E_1)$.

As for $\Phi_{A,2}$, we may assume that $0\in \rho(A(t))$ for each $t\in [0,T]$ (otherwise   we    choose $\omega_0$ with~$\omega_0\in \rho(A(t))$ for each $t\in [0,T]$ and replace $A$  and $\ww(t)$ by $A-\omega_0$ and $\ww(t)e^{\omega_0 t}$, respectively). We then infer from \eqref{Gen} and \cite[II.Lemma 4.3.1]{LQPP} that there exists  a constant  $c>0$ such that
\begin{equation}\label{wAs}
\|w_A(t,s)\|_{\mathcal{L}(E_0)}\le c (t-s)^{\rho-1}\,,\quad 0\le s< t\le T\,.
\end{equation}
We next define the function $\mathcal{F}:[0,T]\to\kL(E_0,E_1)$ by setting 
\[
\mathcal{F}(t):=\int_0^t A(s)^{-1} e^{(t-s)A(s)}w_A(s,0)\,\rd s,\quad 0\leq t\leq T.
\]
Using \eqref{wAs}, it follows that~$\mathcal{F}\in C([0,T],\kL(E_0))$ and, for each $x\in E_0$, that~$\mathcal{F}(\cdot )x\in C^1((0,T],E_0)$ 
with 
\[
\frac{\rd }{\rd t}\big(\mathcal{F}(t)x\big)=A(t)^{-1}w_A(t,0)x+\int_0^t e^{(t-s)A(s)}w_A(s,0)x\,\rd s,\quad 0<t\leq T.
\]
Therefore, given $x\in E_0$,  integration by parts leads us to
\begin{align}
\Phi_{A,2}x&=\int_0^T \ww(t) \Big[\frac{\rd }{\rd t}\big(\mathcal{F}(t)x\big)-A(t)^{-1}w_A(t,0)x\Big]\,\rd t\nonumber\\
&=\ww(T)\mathcal{F}(T)x-\int_0^T \ww(t)A(t)^{-1}w_A(t,0)x\,\rd t-\int_0^T \ww'(t) \mathcal{F}(t)x\,\rd t\,.\label{p3}
\end{align}

We now readily obtain that $\Phi_{A,2}\in \mathcal{L}(E_0,E_1)$ and hence $\Phi_A\in \mathcal{L}(E_0,E_1)$.\medskip

\noindent{\bf (ii)} Let us now also impose \eqref{comp}-\eqref{b0}. 
Combining \eqref{desc1}-\eqref{p1}, we have  
\begin{align}\label{p4}
A(0)\Phi_{A}+\ww(0) =K + A(0)\Phi_{A,2}\,,
\end{align}
where 
$$
K:=\ww(T)\,e^{TA(0)}-\int_0^T\ww'(t)\,e^{tA(0)}\,\rd s\,.
$$
Note that \cite[II.Lemma~5.1.3]{LQPP} and~\eqref{comp} imply for an arbitrary $\theta\in (0,1)$ that 
\begin{equation}\label{p2}
K\in \mathcal{L}(E_0,E_\theta)\subset \mathcal{K}(E_0)\,,
\end{equation}
where $\mathcal{K}$ refers to compact operators.
Additionally, since   $e^{t A(s)}\in\mathcal{L}(E_0,E_1)\subset \mathcal{K}(E_0)$ for each $t>0$ and~$s\in[0,T]$, we infer from \eqref{aA} that
\begin{align}\label{aAcomp}
a_A(t,s)=e^{(t-\ve-s) A(s)} e^{\ve A(s)}\in\mathcal{K}(E_0,E_1)\,,\quad 0\le s<t\le T\,,
\end{align}
by choosing $\ve\in (0,t-s)$.
In view of definition \eqref{wA}, we then get 
$$
k_A(t,s)=[A(t)-A(s)]a_A(t,s)\in\mathcal{K}(E_0)\,,\quad 0\le s<t\le T\,.
$$
 Since $\mathcal{K}(E_0)$ is closed in $\mathcal{L}(E_0)$, we may use Riemann sums to derive that 
 $k_A*k_A(t,s) \in\mathcal{K}(E_0)$ and inductively that $k_A*\ldots *k_A(t,s) \in\mathcal{K}(E_0)$ for $0\le s<t\le T$. 
 Now,  \cite[II.Lemma~4.3.1]{LQPP} ensures that 
 \begin{align}\label{wAcomp}
w_A(t,s)\in\mathcal{K}(E_0)\,,\quad 0\le s<t\le T\,,
\end{align}
 using again the fact that $\mathcal{K}(E_0)$ is a closed subspace of $\mathcal{L}(E_0)$. It then  follows \eqref{p3} and~\eqref{wAcomp}, since~$A(t)^{-1}\in \mathcal{L}(E_0,E_1)$, $t\in[0,T]$,   that
$\Phi_{A,2}\in\mathcal{K}(E_0,E_1)$, 
and, together with \eqref{p4}-\eqref{p2}, we  conclude that
$$
A(0)\Phi_{A}+\ww(0)\in\mathcal{K}(E_0)\,.
$$ 
Since $\ww(0)\not=0$ by \eqref{b0}, the Riesz-Schauder theorem implies that~${A(0)\Phi_{A} \in \mathcal{L}(E_0)}$ is a Fredholm operator with index zero. Due to 
$A(0)^{-1} \in \mathcal{L}is(E_0,E_1)$, we obtain that~$\Phi_{A} \in \mathcal{L}(E_0,E_1)$ is as well a Fredholm operator with index zero.
\end{proof}

\begin{rem}\label{R7}
One may consider as well a slightly more general nonlocal condition of the form
$$
\aw u(T)+\int_0^T  \ww(t)\, u(t)\,\rd t=M
$$
with a given $\aw\in \R$ (e.g. see~\cite{Dokuchaev19}),  which leads to the operator
 \begin{equation}\label{phiA0}
\tilde \Phi_A=\aw  U_{A}(T,0)+\int_0^T\ww(t)\,U_{A}(t,0)\,\rd t\,.
\end{equation}
Then Proposition~\ref{P1} and its proof remain valid verbatim as is easily seen.
\end{rem}

Before continuing let us point out that Proposition~\ref{P1} (together with Remark~\ref{R7}) entail a nontrivial extension of \cite[Corollary 2.3,Theorem 2.4]{SchmitzWalkerJDE} to the semilinear problem
\begin{equation}\label{Ex}
u'=A(t)u+f(u)\,,\quad t\in(0,T]\,,\qquad \aw \, u(T)+\int_0^T  \ww(t)\, u(t)\,\rd t=M\,,
\end{equation}
with time-dependent dependent operator $A(t)$. To make this precise, we assume that there are
\begin{subequations}\label{DD}
\begin{equation}\label{A1Yxxxx}
 \ell>0\,,\qquad 0<\gamma\le 1\,,\qquad 0<\xi<\frac{1}{\ell+1}\,,
\end{equation}
such that $f:E_\xi\cap \mathbb{B}_{E_0}(0,r_0)\to E_\gamma$ for some $r_0>0$. Moreover, we assume that there exists a constant $c(r_0)>0$ such that,
 for $v,w\in E_\xi\cap \mathbb{B}_{E_0}(0,r_0)$,
\begin{equation}\label{F2xx}
\|f(v)-f(w)\|_{E_\gamma}\le c(r_0) \left(\left[\|v\|_{E_\xi}^\ell+\|w\|_{E_\xi}^\ell\right] \|v-w\|_{E_\xi}+ \left[\|v\|_{E_\xi}^{\ell+1}+\|w\|_{E_\xi}^{\ell+1}\right] \|v-w\|_{E_0}\right)
\end{equation}
and
\begin{equation}\label{F2xxx}
\qquad f(0)=0\,,
\end{equation}
\end{subequations}
where $E_\theta:=(E_0,E_1)_\theta$ for $\theta\in (0,1)$ with some admissible interpolation functor $(\cdot,\cdot)_\theta$.

\begin{cor}\label{C1}
Suppose \eqref{Gen}, \eqref{b}, \eqref{comp}, \eqref{b0}, and $\mathrm{ker}(\tilde\Phi_A)=\{0\}$ with $\tilde\Phi_A$ defined in~\eqref{phiA0}. Moreover, assume~\eqref{DD}.  
Then, there is $m_0>0$ such that, for each $M\in E_1$ with $\|M\|_1\le m_0$, problem~\eqref{Ex} has a unique  solution
$$
u\in C^1\big((0,T],E_0\big)\cap C\big((0,T], E_1\big)\cap C\big([0,T],E_0\big)\,.
$$
with $\lim_{t\to 0}t^\xi\|u(t)\|_{E_\xi}=0$.
\end{cor}

\begin{proof}
Using the invertibility of the operator $\tilde \Phi_A\in\mathcal{L}(E_0,E_1)$ stated in Proposition~\ref{P1} and Remark~\ref{R7}, the proof is the same as in 
 \cite[Corollary 2.3,Theorem 2.4]{SchmitzWalkerJDE}.
\end{proof}

In fact, Corollary~\ref{C1} follows also (with minor adaptions) as a special case from Theorem~\ref{MT1} for the borderline value $\alpha=0$, 
which is excluded for the quasilinear problem but admissible for the  semilinear problem (where one may choose $\alpha = \beta$).

\subsection{Lipschitz Continuity of $\Phi_A$ with respect to $A$}

Having established conditions for the invertibility of  the operator $\Phi_A\in\mathcal{L}(E_0,E_1)$ for a given $A\in C^\rho([0,T],\mathcal{H}(E_1,E_0))$, we next investigate Lipschitz properties of the mapping $A\mapsto \Phi_A$. 
Let us first point out the following:

\begin{rem}\label{R5} 
The Lipschitz continuity of the mapping $A\mapsto \Phi_A$ in suitable interpolation spaces follows from~\cite[II.Lemma~5.1.4]{LQPP}.
Indeed, if $E_\theta=(E_0,E_1)_\theta$ for some admissible interpolation functors~$(\cdot,\cdot)_\theta$ of exponent $\theta\in (0,1)$, then
$$
\|\Phi_{A}-\Phi_{B}\|_{\mathcal{L}(E_\alpha,E_\beta)}\le c(T) \|\ww\|_\infty  \| A-B\|_{C([0,T],\mathcal{L}(E_1,E_0))}
$$
provided that $0<\alpha\le 1$ and $0\le\beta<1$. However, this result is insufficient for our purpose since we shall need the estimate for the limiting cases $\alpha=0$ and $\beta=1$, i.e. in~$\mathcal{L}(E_1,E_0)$, which is not included in~\cite[II.Lemma~5.1.4]{LQPP}.
\end{rem}

To state our result, let  $\rho\in (0,1)$ be fixed. Given  $\kappa\ge 1$,~$\omega>0$, and $\eta>0$, we denote by \begin{subequations}\label{stern}
\begin{equation}
\mathcal{A}=\mathcal{A}(\kappa,\omega,\eta)\subset C^\rho\big([0,T],\mathcal{L}(E_1,E_0)\big)
\end{equation}   a set satisfying\footnote{Here,~${\mathcal{H}(E_1,E_0;\kappa,\omega)}$ denotes the class of all operators $A\in\mathcal{L}(E_1,E_0)$ such 
that $\omega-A$ is an
isomorphism from~$E_1$ onto~$E_0$ and
$$
\frac{1}{\kappa}\,\le\,\frac{\|(\mu-A)z\|_{0}}{\vert\mu\vert \,\| z\|_{0}+\|z\|_{1}}\,\le \, \kappa\ ,\qquad {\rm Re\,}\mu\ge \omega\ ,\quad z\in E_1\setminus\{0\}\,.
$$
Observe (see~\cite{LQPP}) that $$\mathcal{H}(E_1,E_0)=\bigcup_{\omega>0\,,\,\kappa\ge 1} \mathcal{H}(E_1,E_0;\kappa,\omega)\,.$$}
\begin{equation}
A\in \mathcal{H}(E_1,E_0;\kappa,\omega)\,,\quad  A\in\mathcal{A}\,,
\end{equation} 
and 
\begin{equation} 
\vert\vert\vert A\vert\vert\vert_\rho\le \eta\,,\quad  A\in\mathcal{A}\,,
\end{equation}
\end{subequations}
where $\vert\vert\vert \cdot\vert\vert\vert_\rho$ denotes the norm in $C^\rho([0,T],\mathcal{L}(E_1,E_0))$.

Let us point out that for each $B\in C^\rho\big([0,T],\mathcal{H}(E_1,E_0)\big)$ there are  $\kappa\ge 1$,~$\omega>0$, and~${\eta>0}$ 
such that $\mathcal{A}(\kappa,\omega,\eta)$ is a neighborhood of $B$ in $C^\rho([0,T],\mathcal{L}(E_1,E_0))$  (see \cite[I.Theorem~1.3.1, I.Corollary~1.3.2]{LQPP}).\\

Our main result regarding the Lipschitz continuity of $A\mapsto\Phi_A$ is then:

\begin{thm}\label{T1}
Let $\ww\in C^1([0,T])$ and $\rho\in (0,1)$. Then the  mapping
$$
\Phi: C^\rho\big([0,T],\mathcal{H}(E_1,E_0)\big)\longrightarrow \mathcal{L}(E_0,E_1)\,,\quad A\mapsto \Phi_A\,,
$$
is locally Lipschitz continuous. In fact, $\Phi:\mathcal{A}\to\mathcal{L}(E_0,E_1)$ is uniformly Lipschitz continuous, where~$\mathcal{A}$ satisfies \eqref{stern}.
\end{thm}

The proof of Theorem~\ref{T1} relies on a careful analysis of the construction of the evolution operator~$U_A$ from~\cite{LQPP} (see \eqref{UA}-\eqref{wA}), for which we require some preparatory steps. \medskip

In the following, we denote by $c$ (or $c_i$, $i\in\N$) constants depending on $\rho,$ $\omega,$ $\kappa,$ $\eta$, and $T$,  whose value may change from one occurrence to another.

\begin{lem}\label{A}
Given $\kappa\ge 1$ and $\omega>0$, there is $M>0$ depending on these parameters such that
\begin{equation}\label{e1s}t^j \|e^{t A}-e^{tB}\|_{\mathcal{L}(E_0,E_j)}\le M \|A-B\|_{\mathcal{L}(E_1,E_0)}\,,\qquad t\in [0,T]\,,\quad j=0,1\,,
\end{equation}
for $A,B\in \mathcal{H}(E_1,E_0;\kappa,\omega)$.
\end{lem}

\begin{proof}
It follows from~\cite[I.Proposition~1.4.1, I.Corollary~1.4.3]{LQPP} that there are $\vartheta\in (0,\pi/2)$ and $c_0>0$ such that
\begin{equation}\label{b1}
\omega+\Sigma_\vartheta\subset \rho(A) \qquad\text{and}\qquad (1+\vert\lambda\vert)^{1-j}\|  (\lambda-(A-\omega))^{-1}\|_{\mathcal{L}(E_0,E_j)}\le c_0\,,\quad \lambda\in \Sigma_\vartheta\,,
\end{equation}
for $j=0,1$ and every $A\in \mathcal{H}(E_1,E_0;\kappa,\omega)$. Here, $\rho(A)$ is the resolvent set of $A$ and 
\begin{equation*}
\Sigma_\vartheta:=\{z\in\mathbb{C}\,:\, |{\rm arg\,}(z)|\leq \vartheta+\pi/2\}\cup\{0\}\,.
\end{equation*}

Since $e^{tA}=e^{t\omega}e^{t(A-\omega)}$, we may assume without loss of generality that $A$ and $B$ satisfy \eqref{b1} with~$\omega=0$.
Moreover, for each $t>0$, it holds that 
\begin{equation}\label{b2}
e^{tA}-e^{tB}=\frac{1}{2\pi i}\int_{\Gamma_t} e^{t\lambda}\big[ (\lambda-A)^{-1}-(\lambda-B)^{-1}\big]\,\rd \lambda\,,\quad t>0\,,
\end{equation}
 where $\Gamma_t=\Gamma_t^++S_t+\Gamma_t^-$ is the piecewise smooth path with
\[
\Gamma_t^\pm:=\{re^{\pm i(\vartheta+\pi/2)}\,:\, t^{-1}\leq r<\infty\}\qquad\text{and}\qquad S_t:=\{t^{-1}e^{i\varphi}\,:\, |\varphi|\leq\vartheta+\pi/2\}\,.
\]
Since 
$$
(\lambda-A)^{-1}-(\lambda-B)^{-1}=(\lambda-A)^{-1}(A-B)(\lambda-B)^{-1},\quad\lambda\in\Sigma_\vartheta\,, 
$$
for $j\in\{0,1\}$, $\lambda\in\Sigma_\vartheta$, and $t>0$  we have, due to~\eqref{b1}, that
\begin{align}\label{b4}
\| e^{t\lambda}\big[ (\lambda-A)^{-1}-(\lambda-B)^{-1}\big]\|_{\mathcal{L}(E_0,E_j)}
&\le c_0^2\,e^{t\mathrm{Re}\,\lambda}\,\vert\lambda\vert^{j-1}\, \|A-B\|_{\mathcal{L}(E_1,E_0)}\,.
\end{align}
for $\lambda\in \dot\Sigma_\vartheta$ and $t>0$. Now, the assertion follows by estimating \eqref{b2} with the help of~\eqref{b4}  (see the proof of~\cite[II.Lemma~4.1.1]{LQPP} for details).
\end{proof}

Given $\alpha\in \R$, let~$\mathfrak{K}(E_0,\alpha)$ be the  Banach space  consisting of those continuous 
function $$k:\{(s,t)\,:\, 0\leq s<t\leq T\}\to \mathcal{L}(E_0)$$ satisfying
$$
\|k\|_{\mathfrak{K}(E_0,\alpha)}:=\sup_{0\le s<t\le T} (t-s)^\alpha \|k(t,s)\|_{\mathcal{L}(E_0)}<\infty\,.
$$

\begin{lem}\label{BC}
Defining $k_A(t,s)$ by \eqref{wA}, it holds that $k_A\in\mathfrak{K}(E_0,1-\rho)$ for $A\in\mathcal{A}$. Moreover, there is $\ell_0>0$ (depending on $\mathcal{A}$) such that
\begin{equation}\label{e3}
\|k_A\|_{\mathfrak{K}(E_0,1-\rho)}\le \ell_0
\end{equation}
and
\begin{equation}\label{e4}
\|k_A-k_B\|_{\mathfrak{K}(E_0,1-\rho)}\le \ell_0\vert\vert\vert A-B\vert\vert\vert_\rho
\end{equation}
for $A,B\in\mathcal{A}$.
\end{lem}

\begin{proof}
According to \cite[II.Lemma~4.2.1]{LQPP} there is $c_1>0$ such that
\begin{equation}\label{e5}
t^j\left(\|e^{t A(s)}\|_{\mathcal{L}(E_0,E_j)}+\|e^{t B(s)}\|_{\mathcal{L}(E_0,E_j)}\right)\le c_1\,,\qquad t\in (0,T]\,,\quad s\in [0,T]\,,\quad j=0,1\,.
\end{equation}
This together with assumption~\eqref{stern} implies~\eqref{e3}.

 Moreover, using~Lemma~\ref{A} and~\eqref{e5} we derive
\begin{align*}
\|(k_A-k_B)(t,s)\|_{\mathcal{L}(E_0)}&\le   \|A(t)-A(s)-B(t)+B(s)\|_{\mathcal{L}(E_1,E_0)}\,\|e^{(t-s)A(s)}\|_{\mathcal{L}(E_0,E_1)}\\
&\quad + \|B(t)-B(s)\|_{\mathcal{L}(E_1,E_0)}\,\|e^{(t-s)A(s)}-e^{(t-s)B(s)}\|_{\mathcal{L}(E_0,E_1)}\\
&\le c_1 \vert\vert\vert A-B\vert\vert\vert_\rho\, (t-s)^{\rho-1}+ \eta  M\,\|A(s)- B(s)\|_{\mathcal{L}(E_0,E_1)}\, (t-s)^{\rho-1}\\
&\le \ell_0 \vert\vert\vert A-B\vert\vert\vert_\rho\,(t-s)^{\rho-1}
\end{align*}
for $0\le s<t\le T$, hence~\eqref{e4}.
\end{proof}

\begin{rem}\label{R10}
It is worth pointing out that the Hölder seminorm of $A-B$ in \eqref{e4} (and hence throughout the subsequent analysis) appears in order to control the term $$\|A(t)-A(s)-B(t)+B(s)\|_{\mathcal{L}(E_1,E_0)}$$ in the last estimate of the proof of Lemma~\ref{BC}.
\end{rem}

We derive now the Lipschitz continuity of the mapping $A\mapsto w_A$.

\begin{prop}\label{D}
Defining $w_A$ by~\eqref{wA}, it holds that $w_A\in \mathfrak{K}(E_0,1-\rho)$ for $A\in\mathcal{A}$. Moreover,  given $\ve>0$  there is a constant $c_2=c_2(\ve)>0$ such that
\begin{equation}\label{e6}
\|w_A-w_B\|_{\mathfrak{K}(E_0,1-\rho)}\le c_2 m e^{(1+\ve)m^{1/\rho}T} \vert\vert\vert A-B\vert\vert\vert_\rho\,,\quad A,B\in\mathcal{A}\,,
\end{equation}
where $m:=2\ell_0\Gamma(\rho)$ with $\ell_0>0$ being the constant from~Lemma~\ref{BC} and $\Gamma$  the Gamma function.
\end{prop}

\begin{proof} 
Clearly, $w_A\in \mathfrak{K}(E_0,1-\rho)$ for $A\in\mathcal{A}$ by~\eqref{wAs}. 
The proof of the Lipschitz property proceeds along the same lines as \cite[II.Section~3.1,Section~3.2]{LQPP}. 
We write
\begin{equation}\label{e7}
w_A-w_B=\sum_{n=1}^\infty \Big(\sum_{j=1}^{n} \underbrace{k_B*\ldots *k_B}_{j-1\text{ times}}*(k_A-k_B)*k_A*\ldots *k_A\Big)\,.
\end{equation}
It is obvious that, for $k\in \mathfrak{K}(E_0,\alpha)$ and $h\in \mathfrak{K}(E_0,\beta)$ with $\alpha,\beta<1$,
$$
\|k*h(t,s)\|_{\mathcal{L}(E_0)}\le \|k\|_{\mathfrak{K}(E_0,\alpha)}\, \|h\|_{\mathfrak{K}(E_0,\beta)}\,\mathsf{B}(1-\alpha,1-\beta)(t-s)^{1-\alpha-\beta}\,,\quad 0\le s<t\le T\,,
$$
where $\mathsf{B}$ denotes the Beta function.
By induction, it then follows that for $k_j\in \mathfrak{K}(E_0,\alpha)$  satisfying~$\|k_j\|_{\mathfrak{K}(E_0,\alpha)}\le \ell_0$, $1\leq j\leq n$,  one has
\begin{align}\label{e9}
\|k_1*\ldots*k_n(t,s)\|_{\mathcal{L}(E_0)}\le \frac{(\ell_0\Gamma(1-\alpha))^n}{\Gamma(n(1-\alpha))} (t-s)^{n(1-\alpha)-1}\,,\qquad n\ge 1\,,\quad 0\le s<t\le T\,. 
\end{align}
Consequently, using~\eqref{e3},~\eqref{e4}, and \eqref{e9}, we deduce for $n\ge 1$ and $0\le s<t\le T$ that
\begin{align*}
&\Bigg\|\sum_{j=1}^{n} \underbrace{k_B*\ldots *k_B}_{j-1}*(k_A-k_B)* k_A*\ldots *k_A (t,s)\Bigg\|_{\mathcal{L}(E_0)}\\
&\qquad\le n\frac{(\ell_0\Gamma(\rho) )^n \vert\vert\vert A-B\vert\vert\vert_\rho}{\Gamma(n\rho)} (t-s)^{n\rho-1}\le \frac{[m(t-s)^{\rho}]^{n-1} }{\Gamma(n\rho)} m(t-s)^{\rho-1}\vert\vert\vert A-B\vert\vert\vert_\rho\, ,
\end{align*}
where we used that $n\le 2^n$. Since 
$$
\sum_{n=1}^{\infty }\frac{[m(t-s)^{\rho}]^{n-1} }{\Gamma(n\rho)}\le c(\rho,\ve)e^{(1+\ve)m^{1/\rho}(t-s)}
$$
for a fixed $\ve>0$ (see the proof of \cite[II.Lemma~3.2.1]{LQPP}), it follows from~\eqref{e7} that
\begin{align*}
\|(w_A-w_B)(t,s)\|_{\mathcal{L}(E_0)}\le c(\rho,\ve)e^{(1+\ve)m^{1/\rho}(t-s)} m(t-s)^{\rho-1}\vert\vert\vert A-B\vert\vert\vert_\rho 
\end{align*}
for $0\le s<t\le T$. This yields the assertion.
\end{proof}

Recalling the formula
$$
U_A=a_A+a_A*w_A\,,
$$
it is now an immediate consequence of~\eqref{wAs}, Lemma~\ref{BC}, and Proposition~\ref{D} that the evolution operator $U_A$ depends continuously on $A$ in $\mathcal{L}(E_0)$
 (note that this is not included in~\cite[II.Lemma~5.1.4]{LQPP} and that it need not be true in $\mathcal{L}(E_0,E_1)$):

\begin{cor}\label{E}
There exists a constant $c_3>0$ such that, for $A,B\in\mathcal{A}$,
\begin{equation}\label{e10}
\|(U_A-U_B)(t,s)\|_{\mathcal{L}(E_0)}\le c_3  \vert\vert\vert A-B\vert\vert\vert_\rho\,,\quad 0\le s\le t\le T\,.
\end{equation}
\end{cor}

We are now in a position to establish the Lipschitz continuity of the operator $A\mapsto \Phi_A$ stated in Theorem~\ref{T1}.

\subsection*{Proof of Theorem~\ref{T1}}

Let $A,B\in \mathcal{A}$.
Since
$U_A(t,s)=e^{\omega(t-s)}U_{A-\omega}(t,s)$
for $\omega\in\R$, we may assume without loss of generality that $\omega=0\in \rho(A(t))\cap\rho(B(t))$ for $t\in [0,T]$ (see ~\eqref{b1}) and, in view of the identity 
$$A(t)^{-1}-B(t)^{-1}=A(t)^{-1}(B(t)-A(t))B(t)^{-1}\,,$$ that
\begin{equation}\label{i1}
\|A(t)^{-1}-B(t)^{-1}\|_{\mathcal{L}(E_0,E_1)}\le c \,\|A(t)-B(t)\|_{\mathcal{L}(E_1,E_0)}\,,\quad t\in [0,T]\,.
\end{equation}
Now, write again 
$\Phi_A=\Phi_{A,1}+\Phi_{A,2}$
and recall from~\eqref{p1}  that 
\begin{align}
\Phi_{A,1}=-\ww(0)A(0)^{-1}+\ww(T)A(0)^{-1}\,e^{TA(0)}-A(0)^{-1}\int_0^T{\ww'(t)}\,e^{tA(0)}\, \rd t\,,\label{i2}
\end{align}
while~\eqref{p3} entails that
\begin{equation}\label{i3}
\begin{aligned}
\Phi_{A,2}&=-\int_0^T \ww'(t) \int_0^t A(s)^{-1} e^{(t-s)A(s)}w_A(s,0)\,\rd s\,\rd t \\
&\quad +\ww(T)\int_0^T A(s)^{-1} e^{(T-s)A(s)}w_A(s,0)\,\rd s-\int_0^T \ww(t)A(t)^{-1}w_A(t,0)\,\rd t\,.
\end{aligned}
\end{equation}
Note that
\begin{align}
 \|A(0)^{-1}e^{tA(0)}-B(0)^{-1}e^{tB(0)}\|_{\mathcal{L}(E_0,E_1)}&\le  \|A(0)^{-1}-B(0)^{-1}\|_{\mathcal{L}(E_0,E_1)} \|e^{tA(0)}\|_{\mathcal{L}(E_0)}\nonumber\\
&\quad+  \|B(0)^{-1}\|_{\mathcal{L}(E_0,E_1)}\|e^{tA(0)}-e^{tB(0)}\|_{\mathcal{L}(E_0)}\nonumber\\
&\le c  \|A(0)-B(0)\|_{\mathcal{L}(E_1,E_0)}\label{i4}
\end{align}
for $t\in [0,T]$ due  to \eqref{i1} and Lemma~\ref{A}, while
\begin{align}
 \|A(s)^{-1} e^{(t-s)A(s)}&w_A(s,0)-B(s)^{-1} e^{(t-s)B(s)}w_B(s,0)\|_{\mathcal{L}(E_0,E_1)}\nonumber\\
&\le  \|A(s)^{-1}-B(s)^{-1}\|_{\mathcal{L}(E_0,E_1)} \|e^{(t-s)A(s)}\|_{\mathcal{L}(E_0)} \|w_A(s,0)\|_{\mathcal{L}(E_0)} \nonumber\\
&\quad +\|B(s)^{-1}\|_{\mathcal{L}(E_0,E_1)} \|e^{(t-s)A(s)}-e^{(t-s)B(s)}\|_{\mathcal{L}(E_0)} \|w_A(s,0)\|_{\mathcal{L}(E_0)}\nonumber
\\
&\quad +\|B(s)^{-1}\|_{\mathcal{L}(E_0,E_1)} \|e^{(t-s)B(s)}\|_{\mathcal{L}(E_0)} \|w_A(s,0)-w_B(s,0)\|_{\mathcal{L}(E_0)}\nonumber\\
&\le c  \big(\|A(s)-B(s)\|_{\mathcal{L}(E_1,E_0)} + \vert\vert\vert A-B\vert\vert\vert_\rho\big) s^{\rho-1}\label{i5}
\end{align}
for $0< s< t\le T$ due to~\eqref{b1}, \eqref{e5}, \eqref{i1},  Lemma~\ref{A}, and Proposition~\ref{D}. 
Recalling~\eqref{b}, we infer from~\eqref{i1}-\eqref{i5} that
$$
\|\Phi_A-\Phi_B\|_{\mathcal{L}(E_0,E_1)}\le c||| A-B|||_\rho
$$
for some constant $c>0$ independent of $A,B\in\mathcal{A}$. This proves Theorem~\ref{T1}.
\qed

\section{Proof of Theorem~\ref{MT1}}\label{Sec3}

In this section we establish the proof of Theorem~\ref{MT1}, but postpone it to the end of this section since we still require some preliminary results.

\subsection{Preliminaries}

Assuming now~\eqref{EQ:A}, we  derive stability estimates for the evolution operators~$U_{A(u)}$ and the operator~$\Phi_{A(u)}$ with respect to $u$ in the interpolation-extrapolation scale.  \\

The following results will play a key role in the subsequent analysis.

\begin{prop}\label{consequences}
Assume~\eqref{EQ:A}. Then, the embddings
\begin{equation}\label{c0}
E_{1+\alpha_0}{\hookrightarrow} E_{1+\alpha}{\hookrightarrow}E_{\alpha_0}{\hookrightarrow}  E_\alpha
\end{equation} 
are dense and compact. Moreover, there are $r\in (0,r_0)$, $\kappa\ge 1$, and $\omega>0$ such that 
\begin{equation}\label{c1}
A(v)\in \mathcal{H}(E_{1+\theta},E_\theta;\kappa,\omega)\,,\qquad v\in \bar{\mathbb{B}}_{E_\beta}(0,r)\,,\quad  \theta\in\{0,\alpha,\alpha_0\}\,.
\end{equation}
In fact, the $E_\theta$-realization $A_\theta(v)$ of $A(v)\in\mathcal{H}(E_1,E_0)$ coincides with $A(v)\vert_{E_{1+\theta}}$ for $\theta\in\{\alpha,\alpha_0\}$.
Furthermore, given $\rho\in(0,1)$ and $u\in \bar{\mathbb{B}}_{C^\rho([0,T],E_\beta)}(0,r)$, there is a unique evolution operator $$U_{A_\theta(u)}(t,s)=U_{A(u)}(t,s)\vert_{E_\theta}\,,\quad 0\le s\le t\le T\,,$$ for $A_\theta(u)$ on $E_\theta$ with regularity subspace $E_{1+\theta}$. Moreover, it holds that
\begin{equation}\label{c2}
\Phi_{A(u)}\in\mathcal{L}is(E_\alpha,E_{1+\alpha})\,,\quad u\in \bar{\mathbb{B}}_{C^\rho([0,T],E_\beta)}(0,r)\,,
\end{equation}
and there is a constant $c_0>0$ (depending only on $r$, $\rho$, $\kappa$, $\omega$, $T$, $\alpha$, and $\beta$) such that
\begin{equation}\label{c3}
\| \Phi_{A(u)}^{-1}-\Phi_{A(\bar u)}^{-1}\|_{\mathcal{L}(E_{1+\alpha},E_\alpha)}\le c_0 \|u-\bar u\|_{C^\rho([0,T], E_\beta)}\,,\quad u, \bar u\in \bar{\mathbb{B}}_{C^\rho([0,T],E_\beta)}(0,r)\,.
\end{equation}
\end{prop}

\begin{proof}
The compact embeddings~\eqref{c0} are consequences of~\eqref{a2} and  \cite[V.Theorem 1.5.1]{LQPP}. Also note from~\eqref{a3} and \cite[V.Theorem 2.1.3]{LQPP} that 
\begin{equation}\label{c4}
A_\alpha(0)=A(0)\vert_{E_{1+\alpha}} \in  \mathcal{H}(E_{1+\alpha},E_\alpha)\,.
\end{equation} 
Since \cite[I.Theorem~1.3.1]{LQPP} entails that $\mathcal{H}(E_{1+\theta},E_\theta)$ is open in $\mathcal{L}(E_{1+\theta},E_\theta)$, we deduce  from~\eqref{a5} that that there exist $r\in (0,r_0)$, $\kappa\ge 1$, and $\omega>0$ such that \eqref{c1} is valid for $\theta=0,\alpha$ (and thus for $\theta=\alpha_0$ as well). 
In particular, it holds that~${\omega-A(v)\in\mathcal{L}is(E_{1+\theta},E_\theta)}$ for $v\in \bar{\mathbb{B}}_{E_\beta}(0,r)$. This implies that the $E_\theta$-realization~$A_\theta(v)$ of~${A(v)\in\mathcal{H}(E_1,E_0)}$ coincides with $A(v)\vert_{E_{1+\theta}}$. 
Moreover,  given  $\rho\in(0,1)$ and $u\in \bar{\mathbb{B}}_{C^\rho([0,T],E_\beta)}(0,r)$, we infer from that~\eqref{a5} and \eqref{c1} that
$$
[t\mapsto A(u(t))]\in C^\rho\big([0,T],\mathcal{H}(E_{1+\theta},E_\theta;\kappa,\omega)\big)\,.
$$
Hence, \cite[II.Corollary~4.4.2]{LQPP} ensures that
there is a unique evolution operator $$U_{A_\theta(u)}(t,s)\,,\quad 0\le s\le t\le T\,,$$ for $A_\theta(u)$ on $E_\theta$ with regularity subspace $E_{1+\theta}$. Clearly, $U_{A_\theta(u)}(t,s)=U_{A(u)}(t,s)\vert_{E_\theta}$.

If $x\in E_{\alpha}$, then
$$
\Phi_{A_\alpha(0)}x=\int_0^T\ww(t)e^{t A_\alpha(0)}\,\rd t\, x=\int_0^T\ww(t)e^{t A(0)}\,\rd t\, x=\Phi_{A(0)}x\,,
$$
since $e^{t A_\alpha(0)}=e^{t A(0)}\vert_{E_\alpha}$. 
Thus, $\mathrm{ker}(\Phi_{A_\alpha(0)})\subset\mathrm{ker}(\Phi_{A(0)})=\{0\}$ by \eqref{a6}. Consequently, we deduce from \eqref{c0}, \eqref{c4}, and Proposition~\ref{P1} that 
\begin{equation}\label{c5}
\Phi_{A(0)}\vert_{E_\alpha}=\Phi_{A_\alpha(0)}\in\mathcal{L}is(E_\alpha,E_{1+\alpha})\,.
\end{equation}
Note from \eqref{a5} that we may assume (making $r$ smaller, if necessary) that
\begin{equation} \label{c6}
\|A(z_1)-A(z_2)\|_{\mathcal{L}(E_{1+\theta},E_\theta)}\le c(r)\|z_1-z_2\|_{E_\beta}\,,\qquad z_1,z_2\in \mathbb{B}_{E_\beta}(0,r)\,,\quad \theta\in\{0, \alpha,\alpha_0\}\,,
\end{equation}
for some constant $c(r)>0$. 
Moreover, given $z_1,\ldots, z_4\in \mathbb{B}_{E_\beta}(0,r)$ (making again $r$ smaller, if necessary) it follows from the mean value theorem and again~\eqref{a5} that
\begin{align}
\|A&(z_1)-A(z_2)-A(z_3)+A(z_4)\|_{\mathcal{L}(E_{1+\alpha},E_\alpha)}\nonumber\\
&\le 
\int_0^1\left\| \p A(z_2+\tau(z_1-z_2))-   \p A(z_4+\tau(z_3-z_4))\right\|_{\mathcal{L}(E_\beta,\mathcal{L}(E_{1+\alpha},E_\alpha))}\,\rd \tau\, \|z_1-z_2\|_{E_\beta}\nonumber\\
&\qquad +\int_0^1  \left\| \p A(z_4+\tau(z_3-z_4)) \right\|_{\mathcal{L}(E_\beta,\mathcal{L}(E_{1+\alpha},E_\alpha))}\,\rd \tau\, \|z_1-z_2-z_3+z_4\|_{E_\beta} \nonumber\\
&\le c(r) \big(\|z_1-z_3\|_{E_\beta}+\|z_2-z_4\|_{E_\beta}\big)\,  \|z_1-z_2\|_{E_\beta}+ c(r) \|z_1-z_2-z_3+z_4\|_{E_\beta}\label{c7}
\end{align}
for some $c(r)>0$.
Therefore,  \eqref{c6}-\eqref{c7} imply that
\begin{align}\label{c8}
\|A(u)-A(\bar u)\|_{C^\rho([0,T],\mathcal{L}(E_{1+\alpha},E_\alpha))}\le c(r)\|u-\bar u\|_{C^\rho([0,T],E_\beta)}\,,\quad u, \bar u\in \bar{\mathbb{B}}_{C^\rho([0,T],E_\beta)}(0,r)\,.
\end{align}
Since $\mathcal{L}is(E_\alpha,E_{1+\alpha})$ is open in $\mathcal{L}(E_\alpha,E_{1+\alpha})$ and since
\[
\Phi: C^\rho\big([0,T],\mathcal{H}(E_{1+\alpha},E_\alpha)\big)\longrightarrow \mathcal{L}(E_\alpha,E_{1+\alpha})\,,\quad A\mapsto \Phi_A
\]
is locally Lipschitz continuous by Theorem~\ref{T1}, we now deduce~\eqref{c2} from~\eqref{c1} (with $\theta=\alpha$),~\eqref{c5}, and \eqref{c8}  by making again $r$ smaller, if necessary. 
Finally,~\eqref{c3} is a consequence of the  local Lipschitz continuity of the inversion map and again \eqref{c8} and Theorem~\ref{T1}.
\end{proof}

For the remainder, we fix $2\rho\in (0,\alpha-\beta)$ and in what follows we denote by $c(r)$, $c_i(r)$,  $i\in\N$, etc. positive constants that depend on $r$ and also $T$, $\rho$, $\kappa$, $\omega$ (from Proposition~\ref{consequences}) and the interpolation parameters.\\

Next, we establish stability results for the evolution operator in interpolation-extrapolation spaces. It is worth emphasizing that one must be careful when interpolating `across 1' and therefore the almost reiteration property is employed.

\begin{lem}\label{L14xx}
Assume~\eqref{EQ:A} and \eqref{A1Y}  and let $r\in (0,r_0)$ be as in Propositions~\ref{consequences}. Let 
\begin{align}\label{h11}
0\le\theta_1<\theta_0\le \alpha
\end{align}
and $\ve\in (0,\gamma-\alpha)$. Then, there exists a  constant $c_1(r)>0$   such that 
\begin{align}\label{h1}
\|U_{A(u)}(t,s)\|_{\mathcal{L}(E_\alpha,E_{1+\theta_1})}\le c_1(r) (t-s)^{\alpha-1-\theta_0}\,,\quad 0\le s<t\le T\,,
\end{align}
and
\begin{align}\label{h4xx}
\|U_{A(u)}(t,s)\|_{\mathcal{L}(E_{\gamma},E_{1+\alpha})}\le   c_1(r) (t-s)^{\gamma-\alpha-\ve-1} \,,\quad  0\le s<t\le T\,,
\end{align}
for $u\in \bar{\mathbb{B}}_{C^\rho([0,T],E_\beta)}(0,r)$.
\end{lem}

\begin{proof}
Recall from Proposition~\ref{consequences} (see also~\eqref{c6}) that 
\begin{subequations}\label{gen}
\begin{align}
A(u)\vert_{E_{1+\theta}}=A_\theta(u)\in C^\rho\big([0,T],\mathcal{H}(E_{1+\theta},E_\theta;\kappa,\omega)\big) 
\end{align}
with
\begin{align}
\|A(u)\|_{C^\rho([0,T],\mathcal{L}(E_{1+\theta},E_\theta))}\le c(r)
\end{align}
\end{subequations}
for $\theta\in\{0, \alpha,\alpha_0\}$ and each $u\in \bar{\mathbb{B}}_{C^\rho([0,T],E_\beta)}(0,r)$. We are thus in a position to apply the stability estimates from~\cite[II.Section~5]{LQPP} for the evolution operator $U_{A(u)}$ since \cite[II.Assumption~(5.0.1)]{LQPP} therein is satisfied.
 To this end, we note from \cite[V.Theorem~1.5.3]{LQPP} that
\begin{align}\label{inj1}
(E_\alpha,E_{1+\alpha})_{1+\theta_0-\alpha}\hookrightarrow E_{1+\theta_1}\,.
\end{align} 
Let now $u,\bar u\in \bar{\mathbb{B}}_{C^\rho([0,T],E_\beta)}(0,r)$. 
Noticing that by construction (see the proof of Proposition~\ref{P1})
\[
U_{A_\alpha(u)}(t,s)=U_{A(u)}(t,s)|_{E_\alpha},\quad 0\leq s\leq t\leq T,
\] 
it follows from \eqref{gen} (with $\theta=\alpha$) and \cite[II.Lemma~5.1.3]{LQPP} that
$$
\|U_{A(u)}(\sigma,s)\|_{\mathcal{L}(E_\alpha,(E_\alpha,E_{1+\alpha})_{1+\theta_0-\alpha})}\le c(r) (\sigma-s)^{\alpha-1-\theta_0}\,,\quad 0\le s<\sigma \le T\,,
$$
and hence, from~\eqref{inj1}, we derive~\eqref{h1}.

Similarly, choosing  $\ve\in (0,\gamma-\alpha)$ it follows from \eqref{gen} (with $\theta=\alpha$) that
$$
\|U_{A(u)}(t,\tau)\|_{\mathcal{L}((E_\alpha,E_{1+\alpha})_{\gamma-\alpha-\ve},E_{1+\alpha})}\le c(r) (t-\tau)^{\gamma-\alpha-\ve-1}\,,\quad 0\le \tau<t \le T\,,
$$
and, noticing that $E_\gamma\hookrightarrow (E_\alpha,E_{1+\alpha})_{\gamma-\alpha-\ve}$ by \cite[V.Theorem~1.5.3]{LQPP},
we derive~\eqref{h4xx}.
 This proves the lemma.
\end{proof}

We use the estimates from the previous lemma to establish the Lipschitz continuous dependence of the evolution operator $U_{A(u)}$ from Proposition~\ref{consequences} on~$u \in \bar{\mathbb{B}}_{C^\rho([0,T],E_\beta)}(0,r)$.

\begin{lem}\label{L14}
Assume~\eqref{EQ:A}  and \eqref{A1Y}, let $r\in (0,r_0)$ be as in Propositions~\ref{consequences}, and choose an arbitrary~$\ve\in(0,\min\{\gamma-\alpha_0,\alpha_0-\alpha\})$.
 Then, there exists a  constant $c_2(r)>0$ such that
\begin{align}\label{E1}
\|U_{A(u)}(t,s)-U_{A(\bar u)}(t,s)\|_{\mathcal{L}(E_{\gamma},E_{1+\alpha})}\le  c_2(r)(t-s)^{\gamma-\alpha-1-\ve}\|u-\bar u\|_{C([0,T],E_\beta)}\,,\quad 0\le s<t\le T\,,
\end{align}
and
\begin{align}\label{E2}
\|U_{A(u)}(\cdot,s)-U_{A(\bar u)}(\cdot,s)\|_{C^\rho([s,T],\mathcal{L}(E_{\alpha},E_\beta))}\le  c_2(r)\|u-\bar u\|_{C([0,T],E_\beta)}\,,\quad s\in [0,T)\,,
\end{align}
for $u, \bar u\in \bar{\mathbb{B}}_{C^\rho([0,T],E_\beta)}(0,r)$.
\end{lem}

\begin{proof}
In order to prove~\eqref{E1} we note from
\cite[V.Theorem~1.5.3]{LQPP} that
\begin{align*}
E_\gamma\hookrightarrow (E_{\alpha_0},E_{1+\alpha_0})_{\gamma-\alpha_0-\ve/2}\,,\qquad 
(E_{\alpha_0},E_{1+\alpha_0})_{1+\alpha-\alpha_0+\ve/2}\hookrightarrow E_{1+\alpha}\,.
\end{align*} 
Now, \eqref{E1} follows from~\eqref{gen} (with $\theta=\alpha_0$) and \cite[II.Theorem~5.1.4]{LQPP}.

As for~\eqref{E2} we consider $u, \bar u\in \bar{\mathbb{B}}_{C^\rho([0,T],E_\beta)}(0,r)$ and choose $\theta_i$  such that $$
\beta+\rho<\theta_2<\theta_1<\theta_0<\alpha-\rho\,.
$$ 
First note from~\eqref{gen} (with $\theta=0$) and \cite[II.Equation~(5.3.8)]{LQPP} that
\begin{align}\label{h3}
\|U_{A(u)}(t,\sigma)-U_{A( u)}(\tau,\sigma)\|_{\mathcal{L}(E_{\theta_2},E_\beta)}\le  c(r) (t-\tau)^{\theta_2-\beta}\,,\quad 0\le \sigma\le \tau\le t\le T\,,
\end{align}
while \cite[II.Lemma~5.1.3]{LQPP} entails that
\begin{align}\label{h4}
\|U_{A(u)}(t,\sigma)\|_{\mathcal{L}(E_{\theta_2},E_\beta)}\le   c(r) \,,\quad  0\le\sigma\le t\le T\,,
\end{align}
with a constant $c(r)>0$.
Let $\bar\theta\in (\theta_2,\theta_1)$. Then, \cite[V.Theorem~1.5.3]{LQPP} ensures the  continuity of the embeddings
\begin{align}\label{uu}
 E_{1+\theta_1}\hookrightarrow (E_1,E_{1+\alpha})_{\bar{\theta}/\alpha}\,,\qquad (E_0,E_\alpha)_{\bar{\theta}/\alpha}\hookrightarrow E_{\theta_2}\,.
\end{align} 
Moreover, since~\eqref{c6} and interpolation imply for $\theta\in (0,1)$ that
\begin{align*}
\|A(u(\sigma))&-A(\bar u(\sigma))\|_{\mathcal{L}((E_1,E_{1+\alpha})_\theta,(E_0,E_\alpha)_\theta)}\\
&\le c \|A(u(\sigma))-A(\bar u(\sigma))\|_{\mathcal{L}(E_1,E_0)}^{1-\theta} \|A(u(\sigma))-A(\bar u(\sigma))\|_{\mathcal{L}(E_{1+\alpha},E_\alpha)}^\theta\\
&\le c(r)\|u(\sigma)-\bar u(\sigma)\|_{E_\beta}\,,
\end{align*}
we infer (taking $\theta=\bar{\theta}/\alpha$ in the previous consideration) from~\eqref{uu} that
\begin{align}\label{h2}
\|A(u(\sigma))-A(\bar u(\sigma))\|_{\mathcal{L}(E_{1+\theta_1},E_{\theta_2})}\le    c(r)\|u-\bar u\|_{C([0,T],E_\beta)}\,,\quad 0\le \sigma\le T\,.
\end{align}
Observe for $0\le s< t\le T$  and $x\in E_1$ that  
$$
U_{A(u)}(t,\cdot)U_{A(\bar u)}(\cdot,s)x\in C([s,t], E_0)\cap C^1((s,t), E_0)
$$ 
with
$$
\frac{\rd}{\rd \sigma} \big(U_{A(u)}(t,\sigma)U_{A(\bar u)}(\sigma,s)x\big)=- U_{A(u)}(t,\sigma)\big[A(u(\sigma))-A(\bar u(\sigma))\big] U_{A(\bar u)}(\sigma,s)x\,, \quad s<\sigma<t\,,
$$
hence
$$
 U_{A(u)}(t,s)x-U_{A(\bar u)}(t,s)x= \int_s^t U_{A(u)}(t,\sigma)\big[A(u(\sigma))-A(\bar u(\sigma))\big] U_{A(\bar u)}(\sigma,s)x\,\rd \sigma \,.
$$
Fix now $0\le s\le \tau\le t\le T$. Then, using the previous identity for $t$ and for $t=\tau$, it follows  from the estimates~\eqref{h1}, \eqref{h3}, \eqref{h4}, and~\eqref{h2} that
\begin{align*}
\|& U_{A(u)}(t,s)-U_{A(\bar u)}(t,s)- U_{A(u)}(\tau,s)+U_{A(\bar u)}(\tau,s)\|_{\mathcal{L}(E_\alpha,E_\beta)}\\
&\le \int_s^\tau \big\|\big[U_{A(u)}(t,\sigma)-U_{ A ( u) }(\tau,\sigma)\big] \big[A(u(\sigma))-A(\bar u(\sigma))\big] U_{A(\bar u)}(\sigma,s)\big\|_{\mathcal{L}(E_\alpha,E_\beta)}\,\rd \sigma\\
&\quad+\int_\tau^t \big\|U_{A(u)}(t,\sigma) \big[A(u(\sigma))-A(\bar u(\sigma))\big] U_{A(\bar u)}(\sigma,s)\big\|_{\mathcal{L}(E_\alpha,E_\beta)}\,\rd \sigma\\
&\le \int_s^\tau \big\|U_{A(u)}(t,\sigma)-U_{A ( u) }(\tau,\sigma)\big\|_{\mathcal{L}(E_{\theta_2},E_\beta)}\\ 
&\hspace{1.15cm} \times\big\|A(u(\sigma))-A(\bar u(\sigma))\big\|_{\mathcal{L}(E_{1+\theta_1},E_{\theta_2})} 
\big\|U_{A(\bar u)}(\sigma,s)\big\|_{\mathcal{L}(E_\alpha,E_{1+\theta_1})}\,\rd \sigma\\
&\quad+\int_\tau^t \big\|U_{A(u)}(t,\sigma)\big\|_{\mathcal{L}(E_{\theta_2},E_\beta)} \big\|A(u(\sigma))-A(\bar u(\sigma))\big\|_{\mathcal{L}(E_{1+\theta_1},E_{\theta_2})}
 \big\|U_{A(\bar u)}(\sigma,s)\big\|_{\mathcal{L}(E_\alpha,E_{1+\theta_1})}\,\rd \sigma\\
&\le  c(r) (t-\tau)^{\theta_2-\beta} \|u-\bar u\|_{C([0,T],E_\beta)}  T^{\alpha-\theta_0}
+c(r) \|u-\bar u\|_{C([0,T],E_\beta)} \big((t-s)^{\alpha-\theta_0}- (\tau-s)^{\alpha-\theta_0}\big)\,.
\end{align*}
Since $\rho\le \min\{\theta_2-\beta,\alpha-\theta_0\}$ we have
$$
(t-\tau)^{\theta_2-\beta}\le T^{\theta_2-\beta-\rho}(t-\tau)^{\rho}\,,\qquad (t-s)^{\alpha-\theta_0}- (\tau-s)^{\alpha-\theta_0}\le (t-\tau)^{\alpha-\theta_0}\le T^{\alpha-\theta_0-\rho}(t-\tau)^{\rho}\,,
$$
and therefore
\begin{align}\label{h7}
\|& U_{A(u)}(t,s)-U_{A(\bar u)}(t,s)- U_{A(u)}(\tau,s)+U_{A(\bar u)}(\tau,s)\|_{\mathcal{L}(E_\alpha,E_\beta)}\le c(r) (t-\tau)^{\rho} \|u-\bar u\|_{C([0,T],E_\beta)}\,.
\end{align}
Finally,  \eqref{gen} (with $\theta=0$), \cite[II.Lemma~5.1.4]{LQPP}, and \eqref{c6} ensure for $0\leq s\le t\leq T$ that
\begin{align}
\|U_{A(u)}(t,s)-U_{A(\bar u)}(t,s)\|_{\mathcal{L}(E_{\alpha},E_\beta)}&\le c(r) (t-s)^{\alpha-\beta}\|A(u)-A(\bar u)\|_{C([0,T],\mathcal{L}(E_1,E_0))}\nonumber\\
&\le  c(r)\|u-\bar u\|_{C([0,T],E_\beta)}\,.\label{h8}
\end{align}
The desired estimate~\eqref{E2} is now a consequence of \eqref{h7} and \eqref{h8}.
\end{proof}

We need one last preparation for which we introduce time-weighted spaces. 
Given $\mu\in \R$   and a Banach space $E$, we denote by $C_{\mu}((0,T],E)$ the Banach space of all functions
$u\in C((0,T],E)$ such that 
$t^{\mu} u(t)\rightarrow 0$ in $E$ as $t\rightarrow 0^+$, equipped with the norm
\begin{equation*}
u\mapsto \|u\|_{C_{\mu}((0,T],E)} := \sup\left\{ t^{\mu}\, \|u(t)\|_E \,:\, t\in (0,T]\right\}\,.
\end{equation*}
These spaces are tailored to capture the singular behavior of the evolution operator $U_{A(u)}$ at $t=0$ within the scale of interpolation spaces.

\begin{lem}\label{psi}
Assume~\eqref{EQ:A} and \eqref{A1Y} and let $r\in (0,r_0)$ be as in Propositions~\ref{consequences}. Set
$$
\Psi(u) g:=\int_0^T\ww(t)\int_0^t U_{A(u)}(t,s) g(s)\,\rd s\,\rd t\,,\qquad g\in C_\nu((0,T],E_\gamma)\,,\quad  u\in \bar{\mathbb{B}}_{C^\rho([0,T],E_\beta)}(0,r)\,.
$$
If $\nu\in [0,1)$, then for all $u,\bar u\in \bar{\mathbb{B}}_{C^\rho([0,T],E_\beta)}(0,r)$ it holds that
$$
\Psi(u)\in\mathcal{L}( C_\nu((0,T],E_\gamma),E_{1+\alpha})
$$ 
and there exists a  constant ${c_3(r)>0}$  such that 
\begin{align*}
\|\Psi(u)\|_{\mathcal{L}( C_\nu((0,T],E_\gamma),E_{1+\alpha})}\le c_3(r)
\end{align*}
and
\begin{align*}
\|\Psi(u)-\Psi(\bar u)\|_{\mathcal{L}( C_\nu((0,T],E_\gamma),E_{1+\alpha})}\le c_3(r)\|u-\bar u\|_{C([0,T],E_\beta)}\,.
\end{align*}
\end{lem}

\begin{proof}
Fix $\ve\in(0,\min\{\gamma-\alpha_0,\alpha_0-\alpha\})$ and let $g\in C_\nu((0,T],E_\gamma)$ and $u, \bar u\in \bar{\mathbb{B}}_{C^\rho([0,T],E_\beta)}(0,r)$. Then, it follows from~\eqref{h4xx} that
\begin{align*}
\|\Psi(u) g\|_{E_{1+\alpha}}& \le  c\int_0^T\int_0^t \|U_{A(u)}(t,s)\|_{\mathcal{L}(E_\gamma,E_{1+\alpha})} \| g(s)\|_{E_\gamma}\,\rd s\,\rd t\\
&\le c(r) \int_0^T\int_0^t (t-s)^{\gamma-\alpha-\ve-1} s^{-\nu}\,\rd s\,\rd t\, \|g\|_{C_\nu((0,T],E_\gamma)}\\
&\le c(r)\mathsf{B}(\gamma-\alpha-\ve,1-\nu)\|g\|_{C_\nu((0,T],E_\gamma)}
\end{align*}
and similarly, using~\eqref{E1},
\begin{align*}
\|(\Psi(u)-\Psi(\bar u)) g\|_{E_{1+\alpha}}& \le  c \int_0^T\int_0^t \|U_{A(u)}(t,s)-U_{A(\bar u)}(t,s)\|_{\mathcal{L}(E_\gamma,E_{1+\alpha})} \| g(s)\|_{E_\gamma}\,\rd s\,\rd t\\
&\le  c(r)\int_0^T\int_0^t (t-s)^{\gamma-\alpha-\ve-1} s^{-\nu}\,\rd s\,\rd t\,\|u-\bar u\|_{C([0,T],E_\beta)} \|g\|_{C_\nu((0,T],E_\gamma)}\\
&\le c(r)\,\mathsf{B}(\gamma-\alpha-\ve,1-\nu)\|u-\bar u\|_{C([0,T],E_\beta)}\|g\|_{C_\nu((0,T],E_\gamma)}
\end{align*}
as claimed.
\end{proof}

At this stage we have all prerequisites to establish Theorem~\ref{MT1}.

\subsection{Proof of Theorem~\ref{MT1}}

Assume~\eqref{EQ:A} and \eqref{AA} and recall (see~\eqref{vdk}) that we want to solve the fixed point problem
$$
u(t)=U_{A(u)}(t,0)\Phi_{A(u)}^{-1}\big(M-\Psi(u) f(u)\big)+\int_0^t U_{A(u)}(t,s) f(u(s))\,\rd s\,,\quad t\in [0,T]\,,
$$
with $M\in E_{1+\alpha}$ given and
$$
\Phi_{A(u)}=\int_0^T\ww(t)\,U_{A(u)}(t,0)\,\rd t\,,\qquad \Psi(u)g = \int_0^T\ww(t)\int_0^t U_{A(u)}(t,s) g(s)\,\rd s\,\rd t\,.
$$
To this end, we shall invoke Banach's fixed point theorem. We fix $\mu$ and $\rho$ such that
\begin{equation}\label{n0}
 (\xi-\alpha)_+<\mu<\frac{1}{\ell+1}\,,\qquad  0<2\rho <\min\left\{2(1-\mu(\ell+1)),\alpha-\beta\right\}\,,
\end{equation} 
and introduce, for a given $L\in (0,\min\{r,1\})$, with $r\in (0,r_0)$  as in Propositions~\ref{consequences}, the complete metric space
$$
X_L:=\left\{u\in C_{\mu}((0,T],E_\xi)\cap C^\rho([0,T],E_\beta)\,:\, \|u\|_{C_{\mu}((0,T],E_\xi)}+ \|u\|_{C^\rho([0,T],E_\beta)}\le L\right\}
$$
equipped with the distance function
$$
d_{X_L}(u,\bar u):=\|u-\bar u\|_{C_{\mu}((0,T],E_\xi)}+ \|u-\bar u\|_{C^\rho([0,T],E_\beta)}\,,\quad u, \bar u\in X_L\,.
$$
It then follows from~\eqref{F2} and \eqref{F2x} that there is a constant $c(r_0)>0$ such that
\begin{equation}\label{n1}
\|f(u)-f(\bar u)\|_{C_{\mu(\ell+1)}((0,T],E_\gamma)}\le c(r_0) L^\ell d_{X_L}(u,\bar u)\,,\quad u,\bar u\in X_L\,,
\end{equation}
and
\begin{equation}\label{n2}
\|f(u)\|_{C_{\mu(\ell+1)}((0,T],E_\gamma)}\le c(r_0) L^{\ell+1} \,,\quad u\in X_L\,.
\end{equation}
In particular, using Lemma~\ref{psi} we have $\Psi(u)f(u)\in E_{1+\alpha}$ for $u\in X_L$ and hence,
in view of Proposition~\ref{consequences} we may define, for given $M\in E_{1+\alpha}$, 
$$
\Xi(u):=\Phi_{A(u)}^{-1}\big(M-\Psi(u)f(u)\big) \in E_\alpha\,,\quad u\in X_L\,.
$$
Note from~\eqref{c3},  \eqref{n2}, and Lemma~\ref{psi} that
\begin{equation}\label{Xi1}
\|\Xi(u)\|_{E_\alpha}\le \|\Phi_{A(u)}^{-1}\|_{\mathcal{L}(E_{1+\alpha},E_\alpha)} \big(\|M\|_{E_{1+\alpha}}+\|\Psi(u)f(u)\|_{E_{1+\alpha}}\big)\le c(r) \big(\|M\|_{E_{1+\alpha}}+L^{\ell+1}\big)
\end{equation}
for $u\in X_L$. Moreover, since 
\begin{align*}
\|\Xi(u) -  \Xi(\bar u)\|_{E_\alpha}
&\le \|\Phi_{A(u)}^{-1}-\Phi_{A(\bar u)}^{-1}\|_{\mathcal{L}(E_{1+\alpha},E_\alpha)} \big(\|M\|_{E_{1+\alpha}} +\|\Psi(u)f(u)\|_{E_{1+\alpha}}\big)\nonumber\\
&\quad +\|\Phi_{A(\bar u)}^{-1}\|_{\mathcal{L}(E_{1+\alpha},E_\alpha)} \big(\|\Psi(u)(f(u)-f(\bar u))\|_{E_{1+\alpha}}+\|(\Psi(u)-\Psi(\bar u))f(\bar u)\|_{E_{1+\alpha}}\big)
\end{align*}
we may invoke~\eqref{c3}, \eqref{n1}, \eqref{n2}, Proposition~\ref{consequences}, and  Lemma~\ref{psi}, to obtain
\begin{align}\label{Xi2}
\|\Xi(u) -  \Xi(\bar u)\|_{E_\alpha}
&\le  c(r)\big(\|M\|_{E_{1+\alpha}}+L^{\ell+1}+L^{\ell}\big) d_{X_L}(u,\bar u) 
\end{align}
for $ u, \bar u\in X_L$. 
Defining
$$
F(u)(t):=U_{A(u)}(t,0)\Xi(u)+\int_0^t U_{A(u)}(t,s)f(u(s))\,\rd s\,,\qquad t\in [0,T]\,,\quad u\in X_L\,,
$$
we claim that $F$ defines a contraction on $X_L$ provided that $L\in (0, \min\{r, 1\})$ and $\|M\|_{E_{1+\alpha}}$ are sufficiently small.

We first check that $F$ is a self-mapping on $X_L$. Let $ u\in X_L$ be fixed.
Then, \eqref{n2}, \eqref{Xi1}, and \cite[II.Lemma~5.1.3]{LQPP} imply that (recall that $\alpha>\beta$)
\begin{align}
\|F(u)(t)\|_{E_\beta}&\le \|U_{A(u)}(t,0)\|_{\mathcal{L}(E_\alpha,E_\beta)}\|\Xi(u)\|_{E_\alpha}   +\int_0^t \|U_{A(u)}(t,s)\|_{\mathcal{L}(E_\gamma,E_\beta)} \|f(u(s))\|_{E_\gamma}\,\rd s \nonumber\\ 
& \le c(r) \big(\|M\|_{E_{1+\alpha}}+L^{\ell+1}\big)
\label{g1}
\end{align}
and 
\begin{align}
\|F(u)(t)\|_{E_1}&\le  c(r) (t^{\alpha-1}+t^{\gamma-\mu(\ell+1)})\big(\|M\|_{E_{1+\alpha}}+L^{\ell+1}\big) 
\label{g1'}
\end{align}
for $t\in[0,T]$.  Similarly, choosing $\alpha_0\in(0,\alpha)$ with $(\xi-\alpha_0)_+<\mu$ (see~\eqref{n0}), $\zeta\in( (\xi-\alpha_0)_+,\mu)$, and $\gamma_0\in (\alpha,\gamma)$, we obtain
\begin{align*}
\|F(u)(t)\|_{E_\xi}&\le c(r)  t^{-(\xi-\alpha_0)_+} \|\Xi(u)\|_{E_\alpha}+c(r)\big(1+t^{1+\gamma_0-\xi-\mu(\ell+1)}\big)L^{\ell+1}\\
&\le c(r) t^{-\zeta}\big(\|M\|_{E_{1+\alpha}}+L^{\ell+1}\big)+c(r)\big(1+t^{1+\gamma_0-\xi-\mu(\ell+1)}\big)L^{\ell+1}
\end{align*}
for $t\in(0,T]$. Since
\begin{equation}\label{mm}
1+\gamma_0-\xi-\mu(\ell+1)>-\mu\,,
\end{equation} 
the latter implies that
$$
 \lim_{t\to 0^+}t^\mu\|F(u)(t)\|_{E_\xi}=0
$$
and
\begin{align}
\sup_{t\in (0,T]}t^\mu\|F(u)(t)\|_{E_\xi}\le c(r)\big(\|M\|_{E_{1+\alpha}}+L^{\ell+1}\big)\,.\label{g1x}
\end{align}
Since \cite[II.Theorem~5.3.1]{LQPP} entails that
$$
\|U_{A(u)}(t,0)-U_{A(u)}(s,0)\|_{\mathcal{L}(E_\alpha,E_\beta)}\le c(r)(t-s)^{\alpha-\beta},
$$
while \cite[II.Equation~(5.3.8)]{LQPP} and $U_{A(u)}(s,s)=1$ yield
$$
\|U_{A(u)}(t,s)-1\|_{\mathcal{L}(E_\alpha,E_\beta)}\le c(r)(t-s)^{\alpha-\beta}\,,
$$
it follows from \eqref{n2} and \eqref{Xi1}  that 
\begin{align}
\|F(u)(t)-F(u)(s)\|_{E_\beta} &\le \|U_{A(u)}(t,0)-U_{A(u)}(s,0)\|_{\mathcal{L}(E_\alpha,E_\beta)} \|\Xi(u)\|_{E_{\alpha}} \nonumber\\
&\qquad +\int_0^s \|U_{A(u)}(t,s)-1\|_{\mathcal{L}(E_\alpha,E_\beta)} \|U_{A(u)}(s,\tau)\|_{\mathcal{L}(E_\gamma,E_\alpha)} \|f(u(\tau))\|_{E_\gamma}\,\rd \tau\nonumber\\
&\qquad + \int_s^t \|U_{A(u)}(t,\tau)\|_{\mathcal{L}(E_\gamma,E_\beta)}  \|f(u(\tau))\|_{E_\gamma}\,\rd \tau \nonumber\\
&\le c(r) (t-s)^{\alpha-\beta}\big(\|M\|_{E_{1+\alpha}}+L^{\ell+1}\big) +c(r)L^{\ell+1} (t-s)^{\alpha-\beta}\int_0^s\tau^{-\mu(\ell+1)}\,\rd \tau\nonumber\\
&\qquad  +c(r) L^{\ell+1} \int_s^t \tau^{-\mu(\ell+1)}\,\rd \tau\nonumber\\
&\le c(r) (t-s)^{\alpha-\beta}\big(\|M\|_{E_{1+\alpha}}+L^{\ell+1}\big)+c(r) (t-s)^{1-\mu(\ell+1)} L^{\ell+1} \nonumber\\
&\le c(r) (t-s)^\rho\big(\|M\|_{E_{1+\alpha}}+L^{\ell+1}\big)\label{g2}
\end{align}
for $0\le s<t\le T$ recalling the choice of $\rho$ in~\eqref{n0}.
Since \eqref{g1'}, \eqref{g2}, and interpolation ensure that  $u\in C((0,T],E_\xi)$, we conclude from \eqref{g1}-\eqref{g2}  that there is $m_0>0$ such that
$F:X_L\to X_L$
for $\|M\|_{E_{1+\alpha}}\le m_0$ provided that $L\in (0,\min\{r,1\})$ is sufficiently small. 

We next check the contraction property and set
$$
F_1(u):= U_{A(u)}(\cdot,0)\Xi(u)\,,\qquad F_2(u):=F(u)-F_1(u)
$$
for $u\in X_L$.
Consider $u, \bar u\in X_L$. Then, using \eqref{E2}, \eqref{h3}, and \eqref{Xi1}-\eqref{Xi2} we deduce that
\begin{align}\label{i8}
\|F_1(u)& -  F_1(\bar u)\|_{C^\rho([0,T],E_\beta)}\nonumber\\
&\le \|U_{A(u)}(\cdot,0)-U_{A(\bar u)}(\cdot,0)\|_{C^\rho([0,T],\mathcal{L}(E_\alpha,E_\beta))} \|\Xi(u)\|_{E_\alpha}\nonumber\\
&\qquad +\|U_{A(\bar u)}(\cdot,0)\|_{C^\rho([0,T],\mathcal{L}(E_\alpha,E_\beta))} \|\Xi(u)-\Xi(\bar u)\|_{E_\alpha}\nonumber\\
&\le c  \big(\|M\|_{E_{1+\alpha}}+L^{\ell+1}+L^{\ell}\big) d_{X_L}(u,\bar u)\,.
\end{align}
In much the same way, since 
$$
\|U_{A(\bar u)}(\cdot,0)\|_{C_\mu((0,T],\mathcal{L}(E_\alpha,E_\xi))}\le c(r)
$$
and
$$
\|U_{A(u)}(\cdot,0)-U_{A(\bar u)}(\cdot,0)\|_{C_\mu((0,T],\mathcal{L}(E_\alpha,E_\xi))}\le c(r)\|u-\bar
u\|_{C([0,T],E_\beta)}\,,
$$ 
 by \cite[II.Lemma~5.1.4]{LQPP} and \eqref{c6}\,,
we derive that
\begin{align}\label{i88}
\|F_1(u) -  F_1(\bar u)\|_{C_\mu((0,T],E_\xi)}
\le c  \big(\|M\|_{E_{1+\alpha}}+L^{\ell+1}+L^{\ell}\big) d_{X_L}(u,\bar u)\,.
\end{align}

As for $F_2$ we use the assumption $\gamma>\alpha$,  \cite[II.Equation~(5.3.8)]{LQPP}, \eqref{E2}, and \eqref{h8}-\eqref{n2}   to obtain for $0\le \tau\le t\le T$ that
\begin{align*}
\|F_2(u)(t)& -  F_2(\bar u)(t)-F_2(u)(\tau) +  F_2(\bar u)(\tau)\|_{E_\beta}\\
& \le c   \int_0^\tau \|U_{A(u)}(t,s)-U_{A(\bar u)}(t,s)-U_{A(u)}(\tau,s)+U_{A(\bar u)}(\tau,s)\|_{\mathcal{L}(E_\alpha,E_\beta)} \|f(u(s))\|_{E_\gamma}\,\rd s\\
&\qquad + c \int_0^\tau \|U_{A(\bar u)}(t,s)-U_{A(\bar u)}(\tau,s)\|_{\mathcal{L}(E_\alpha,E_\beta)} \|f(u(s))-f(\bar u(s))\|_{E_\gamma}\,\rd s\\
&\qquad + c\int_\tau^t \|U_{A(u)}(t,s)-U_{A(\bar u)}(t,s)\|_{\mathcal{L}(E_\alpha,E_\beta)}  \|f(u(s))\|_{E_\gamma}\,\rd s\\
&\qquad+ \int_\tau^t \|U_{A(\bar u)}(t,s)\|_{\mathcal{L}(E_\gamma,E_\beta)} \|f(u(s))-f(\bar u(s))\|_{E_\gamma}\,\rd s\\
&\le c(r) (t-\tau)^\rho \big(L^{\ell+1}+L^{\ell}\big) d_{X_L}(u,\bar u)\,
\end{align*}
that is,
\begin{align}\label{i9}
\|F_2(u)- F_2(\bar u)\|_{C^\rho([0,T],E_\beta)}\le c(r)\big(L^{\ell+1}+L^{\ell}\big) d_{X_L}(u,\bar u)\,.
\end{align}
Finally, \eqref{h8} (with $(\alpha,\beta)$ replaced by $(\gamma,\xi)$),  \eqref{n1}, and \eqref{n2} imply for $0<t\le T$ that
\begin{align*}
\|F_2(u)(t) -  F_2(\bar u)(t)\|_{E_\xi}&\le  \int_0^t \|U_{A(u)}(t,s)-U_{A(\bar u)}(t,s)\|_{\mathcal{L}(E_\gamma,E_\xi)} \|f(u(s))\|_{E_\gamma}\,\rd s \\
&\qquad +\int_0^t \|U_{A(\bar u)}(t,s)\|_{\mathcal{L}(E_\gamma,E_\xi)} \|f(u(s))-f(\bar u(s))\|_{E_\gamma}\,\rd s \\
&\le c(r) \big(1+t^{1+\gamma_0-\xi-\mu(\ell+1)}\big) \big(L^{\ell+1}+L^{\ell}\big) d_{X_L}(u,\bar u)\,,
\end{align*}
 with $\gamma_0\in (\alpha,\gamma)$. Recalling~\eqref{mm}, it thus follows that
\begin{align}\label{i99}
\|F_2(u) -  F_2(\bar u)\|_{C_\mu((0,T],E_\xi)}&\le   c(r) \big(L^{\ell+1}+L^{\ell}\big) d_{X_L}(u,\bar u)\,.
\end{align}
Consequently, we conclude from~\eqref{i8}-\eqref{i99}, making $m_0>0$ and $L$ smaller, if necessary,  that  the mapping~$F:X_L\to X_L$ is a contraction, so that Banach's fixed point theorem entails the existence of a unique  $u\in X_L$ with $u=F(u)$.
Since $\Phi_{A(u)}^{-1}(M+\Psi(u) f(u))\in E_\alpha$ with $\alpha\in (0,1)$ and since~$f(u)\in C_\mu((0,T],E_\xi)$, 
one may argue as in the proof of \cite[Proposition 2.1]{MW_PRSE} to deduce that~$u$ is a solution to~\eqref{EE} with regularity as stated in Theorem~\ref{MT1}. 
This proves the claim.\qed

\medskip

We note that Theorem~\ref{MT1} also applies to semilinearities $f=f(u)$ that may depend on $u$ in a nonlocal manner with respect to time:

\begin{rem}
Note that in the proof of Theorem~\ref{MT1}, we only used the properties~\eqref{n1}-\eqref{n2} for the semilinear part $f$. That is, Theorem~\ref{MT1} still remains valid if~\eqref{AA} is replaced by the assumption that 
$$
f:C_{\mu}((0,T],E_\xi)\cap C^\rho([0,T],E_\beta)\to C_{\mu(\ell+1)}((0,T],E_\gamma)
$$
is a mapping with $f(0)=0$, and there is $c>0$ such that, for each $L\in (0,1)$,
\begin{equation*}
\|f(u)-f(\bar u)\|_{C_{\mu(\ell+1)}((0,T],E_\gamma)}\le c L^\ell \left(\|u-\bar u\|_{C_{\mu}((0,T],E_\xi)}+\|u-\bar u\|_{C^\rho([0,T],E_\beta)}\right)
\end{equation*}
whenever 
$$\|u\|_{C_{\mu}((0,T],E_\xi)}+\|u\|_{C^\rho([0,T],E_\beta)}\le L \qquad \text{and}\qquad \|\bar u\|_{C_\mu((0,T],E_\xi)}+\|\bar u\|_{C^\rho([0,T],E_\beta)}\le L\,,
$$ 
where we assume~\eqref{A1Y} and~\eqref{n0}.
\end{rem}

\section{Applications}\label{Sec5}

We provide two examples for recovering the initial state in quasilinear parabolic equations from time averages. 

\subsection{A Chemotaxis System with Local Sensing} 

 As a first application, we consider a chemotaxis system with local sensing related to the classical Keller–Segel model \cite{KS1971}, recently studied in \cite{AY19,JL24}. 
This system describes the evolution of a cell population with density $u$ in response to a chemical signal with concentration $v$, and reads 
\begin{subequations}\label{Exem1}
\begin{align}
\partial_t u&=\nabla \cdot \left( a(v)\nabla u - \chi(v) u \nabla v \right)\,,&& t\in(0,T)\,,\quad x\in \Omega\label{ex1a}\,,\\
0&=\Delta v +u-v\,,&& t\in(0,T)\,,\quad x\in \Omega\,.\label{ex1b}
\end{align}
Here, $\Omega$ is a smooth bounded domain in $\R^n$ with $n\in\{1,2,3\}$, $T>0$ is a fixed time, $a(v)$ represents the diffusion coefficient of the cells, while $\chi(v)$ denotes 
their sensitivity to the chemical gradient.
 These equations are subject to homogeneous Neumann boundary conditions
\begin{align}\label{ex1c}
 \partial_\nu u=   \partial_\nu v&= 0\,, \quad  t\in(0,T]\,,\quad x\in \partial\Omega\,\,.
\end{align}
Instead of an initial condition we assume that the local cumulative population $M$ over time has been observed; that is,
\begin{align}\label{ex1d}
 \int_0^T u(t,x)\,{\rm d}t=M(x)\,,\quad x\in \Omega\,.
\end{align}
\end{subequations}
For simplicity we take here a constant time weight (though more general time weights $\ww=\ww(t)$ can easily be considered).
We demonstrate that this problem is well-posed, hence the cumulative population identifies a unique initial state $u(0)$.

Again, in order to keep things simple, we impose that 
\begin{equation}\label{Cond:1}
a,\chi\in C^\infty(\R)\qquad\text{and}\qquad a(0)>0\,,
\end{equation}
but less regularity is needed, of course. In the Keller–Segel models, natural choices for $a$ include the function~$a(s)=e^{-\delta s}$ with $\delta>0$   (while $\chi=-a'$); see, e.g.,~\cite{JL24}, which motivates assumption~\eqref{Cond:1}.\\

In order to tackle~\eqref{Exem1}, we write it in the form of \eqref{EE}. For this purpose, we denote by $\Delta_N$ the Neumann-Laplacian. Then, since $v=(1-\Delta_N)^{-1} u$ by~\eqref{ex1b}-\eqref{ex1c} and since
\begin{align*}
\nabla \cdot \left( a(v)\nabla u - \chi(v) u \nabla v \right)&=a(v)\Delta u+(a-\chi)'(v)\nabla v\cdot\nabla u-u\chi(v)\Delta v\\
&=a(v)\Delta u+(a-\chi)'(v)\nabla v\cdot\nabla u-\chi(v)uv- \chi(v)u^2\,,
\end{align*}
we see that  \eqref{Exem1} can formally be written in the form of \eqref{EE}
$$
u'=A(u)u+f(u)\,,\quad t\in (0,T)\,,\qquad \int_0^T u(t)\,\rd t=M\,,
$$
by setting
\begin{equation}
\begin{aligned}\label{P}
A(u)w&:=a\big((1-\Delta_N)^{-1}u\big)\Delta w+(a-\chi)'(1-\Delta_N)^{-1}u\nabla w\cdot \nabla\big((1-\Delta_N)^{-1} u\big)\\
&\qquad-\chi\big((1-\Delta_N)^{-1} u\big)w(1-\Delta_N)^{-1}u
\end{aligned}
\end{equation}
and
\begin{equation}\label{Pf}
f(u):=- \chi\big((1-\Delta_N)^{-1} u\big)u^2\,.
\end{equation}
We will see that in this form the assumptions of Theorem~\ref{MT1} can be verified in order to get the following result:

\begin{thm}\label{T:Exem1}
Assume \eqref{Cond:1} and choose $s\in(2,5/2)$. 
Then, there exists a constant~$m_0>0$ such that for each $M\in H^s(\Omega)$  with $\partial_\nu M=0$ on $\partial\Omega$ and $\|M\|_{H^s}\leq m_0$ 
there exists a unique  solution to~\eqref{Exem1} satisfying
\begin{align*}
u&\in C^{\min\{s/2-1, 1-2\mu\}}\big([0,T],L_2(\Omega)\big)\cap C\big([0,T],H^{s-2}(\Omega)\big)\cap C\big((0,T],H^2(\Omega)\big)\cap C^{1}\big((0,T],L_2(\Omega)\big)
\end{align*}
such that $\lim_{t\to0^+}t^\mu \|u(t)\|_{H^{s/4}}=0$ for $ 1-s/4<\mu <1/2$.
\end{thm}

\begin{proof}
Set $E_0:=L_2(\Omega)$ and 
\[
E_1:=H^2_N(\Omega):=\{v\in H^{2}(\Omega) \,:\, \p_\nu v=0 \text{ on } 
 \partial\Omega\}\,.\] 
 Then, $E_1$ is compactly and densely embedded into $E_0$ and $1-\Delta_N\in \mathcal{L}is(H^{2}_N(\Omega),L_2(\Omega))$. 
 Thus,~$A$ from~\eqref{P} is a well-defined mapping $A:E_0\to\mathcal{L}(E_1,E_0)$ with $A(0)=a(0)\Delta_N\in \mathcal{H}(E_1,E_0)$. It then follows from \cite[Theorem~7.1; Equation (7.5)]{Amann_Teubner} and \cite[Theorem 5.5.1]{Tr78} that the interpolation-extrapolation scale generated by $(E_0,1-A(0))$ and  the complex interpolation functor~$[\cdot,\cdot]_\theta$  is given by
\begin{equation}\label{Hs}
E_\theta\doteq H_{N}^{2\theta}(\Omega):=\left\{\begin{array}{ll} \{v\in H^{2\theta}(\Omega) \,:\, \p_\nu v=0 \text{ on } 
 \partial\Omega\}\,, &3/2<2\theta< 7/2 \,,\\[3pt]
	 H^{2\theta}(\Omega)\,, & 0\le 2\theta<3/2\,.\end{array} \right.
\end{equation}
Note that
 \begin{equation}\label{reg:10}
  A\in   C^\infty\big(E_0,\kL(E_{1+\theta},E_\theta)\big)\,,\qquad \theta\in [0,1/4)\,.
  \end{equation}
 Indeed, if $F\in C^\infty(\R)$, it is straightforward to prove that the Nemitskii operator $F(\cdot):=[u\mapsto F(u)]$ satisfies 
  \begin{equation}\label{Nemop}
F(\cdot)\in C^\infty(E_1,E_1)
  \end{equation}
and is uniformly Lipschitz continuous on  bounded subsets of $E_1$.
  Moreover,  the bilinear pointwise multiplications   
\begin{subequations}\label{bilmult}
 \begin{align} 
  & H^2(\Omega)\bullet H^{2\vartheta}(\Omega)\longrightarrow H^{2\vartheta}(\Omega)\,,\quad 0\le 2\vartheta\le 2\,,\label{bilmult1}\\
 & H^1(\Omega)\bullet H^{1}(\Omega)\longrightarrow H^{2\theta}(\Omega)\,,\quad 0<2\theta<\frac{1}{2}\,,\label{bilmult2}
  \end{align}
\end{subequations} 
are both  continuous, where \eqref{bilmult1} is due to the fact that (see \cite[Theorem~2.1]{AmannMult}) $$[u\mapsto au]\in\mathcal{L}\big(L_2(\Omega)\big)\cap\mathcal{L}\big(H^2(\Omega)\big)\,,\quad a\in H^2(\Omega)\,,$$  and interpolation, while \eqref{bilmult2} is due to \cite[Theorem~4.1]{AmannMult}. The regularity properties~\eqref{reg:10} follow now from~\eqref{Hs},~\eqref{Nemop}–\eqref{bilmult}, and the fact that (see \cite[Theorem 5.5.1]{Tr78})
\begin{equation}\label{isomex1}
 1-\Delta_N\in \mathcal{L}is\big(H^{2\theta}_N(\Omega),H^{2\theta-2}_N(\Omega)\big)\,,  \qquad 2\leq 2\theta<7/2\,.
  \end{equation}
As for the semilinear part we point out that, due to the quadratic term $u^2$, the term $f(u)$ cannot be incorporated in the quasilinear structure without violating property~\eqref{reg:10}. 
  Instead, recalling that~$s\in(2,5/2)$, we may choose exponents $\alpha,\alpha_0,\gamma,\xi\in(0,1)$ such that
  \begin{equation}\label{exponents}
  0=\beta<\alpha:=\frac{s}{2}-1<\alpha_0<\gamma<\frac{1}{4}\qquad\text{and}\qquad \frac{1}{2}<\xi:= \frac{s}{4}<\frac{s-1}{2}<\frac{3}{4}\,.
  \end{equation}
  Then, using the continuity of pointwise multiplication (see \cite[Theorem~4.1]{AmannMult}) 
 \begin{align*} 
   H^2(\Omega)\bullet H^{2\xi}(\Omega)\bullet H^{2\xi}(\Omega)\longrightarrow H^{2\gamma}(\Omega)\,,
  \end{align*}
we infer that
\begin{align*}
\|f(u)-f(w)\|_{H^{2\gamma}}&\le c  \|\chi((1-\Delta_N)^{-1}u)-\chi(  (1-\Delta_N)^{-1}w)\|_{H^2}\|u\|_{H^{2\xi}}^2\\
&\qquad+\|\chi((1-\Delta_N)^{-1}w)\|_{H^2}\|u+w\|_{H^{2\xi}}\|u-w\|_{H^{2\xi}}\,. 
\end{align*}
Taking~\eqref{Hs}, \eqref{Nemop}, and~\eqref{isomex1} into account, we can find for a  given $r_0>0$ a constant~$c=c(r_0)>0$ such that, for $v,w\in E_\xi\cap \mathbb{B}_{E_\beta}(0,r_0)$,
\begin{equation}\label{F2ex1}
\|f(u)-f(w)\|_{E_\gamma}\le c(r_0) \big(\|u\|_{E_\xi}^{2}\|v-w\|_{E_\beta}+\big[\|u\|_{E_\xi}+\|w\|_{E_\xi}\big] \|u-w\|_{E_\xi}\big)\,.
\end{equation}
Finally, writing the operator
  \[
\Phi_{A(0)}=\int_0^Te^{t a(0)\Delta_N}\,{\rm d}t  \in\mathcal{K}(L_2(\Omega))
  \] 
in a Fourier series expansion in $L_2(\Omega)$ with respect to a orthonormal basis of eigenfunctions of~$\Delta_N$, the injectivity assumption \eqref{a6} is easily verified (since all eigenvalues of $a(0)\Delta_N$ are nonnegative); see \cite[Proposition 6.1]{SchmitzWalkerJDE}.
Consequently, gathering~\eqref{reg:10} and~\eqref{exponents}-\eqref{F2ex1}, we see that the assumptions of Theorem~\ref{MT1} are satisfied, and the proof of Theorem~\ref{T:Exem1} is complete.
\end{proof}

\begin{rems}
{\bf (a)} Clearly, we did not strive for the most general result when stating Theorem~\ref{T:Exem1}. In fact, less assumptions on the data are required and it is also possible to consider the problem in an  $L_p$-setting with $p>2$ what can be useful for other nonlinearities. 
In this case, one derives the injectivity of $\Phi_{A(0)}$ from the $L_2$-case (see \cite[Section~6]{SchmitzWalkerJDE} or the subsequent Subsection~\ref{S42} for details). 

{\bf (b)} Instead of the chemotaxis model, one may also consider other reaction-diffusion systems, e.g. of the form
\begin{align*}
\partial_t u_i &=\mathrm{div}\big(d_i(T)\nabla u_i)+f_i(u_1,...,u_N,T)\,,\qquad t\in (0,T)\,,\quad x\in \Omega\,,\quad 1\le i\le N\,,\\
-\Delta T&=g(u_1,...,u_N)\,,\qquad t\in (0,T)\,,\quad x\in \Omega\,,
\end{align*}
involving $N$ reactants $u_i$ and a quasi-stationary temperature $T$, 
subject to  suitable boundary conditions (e.g. of Dirichlet type for the temperature) and the conditions
$$
\int_0^T u_i(t,x)\,\rd t=M_i(x)\,,\qquad  x\in \Omega\,,\quad 1\le i\le N\,,
$$
where $f_i$ are polynomials.
\end{rems}

\subsection{Quasilinear Reaction-Diffusion Systems with Nonlocal Coefficients}\label{S42}

We consider the problem of recovering the initial state in nonlinear reaction-diffusion systems with nonlocal diffusion coefficients  (e.g. see~\cite{CaballeroEtal_22,ChipotLovat99,ChipotSiegwart03,FerreiraBorges17}). More precisely, we investigate abstract formulations of the quasilinear prototype problem
\begin{subequations}\label{PA}
\begin{align}
\partial_t u &=\mathcal{A}(x,u)u +f(u)\,,\qquad t\in (0,T)\,,\quad x\in\Omega\,,\label{PA1}\\
\mathcal{B}u&:=(1-\delta)u+\delta\partial_\nu u=0\,,\qquad t\in (0,T)\,,\quad x\in\partial\Omega\,,\label{PA2}\\
\int_0^T &\ww(t)u(t,x)\,\rd t=M(x)\,,\qquad  x\in \Omega\,,\label{PA3}
\end{align}
\end{subequations}
where $\mathcal{A}$ is a  second-order operator on a bounded and smooth domain $\Omega$ of $\R^n$, $n\geq 1$, given by
\begin{align*}
\mathcal{A}(x,u)v=\sum_{j,k=1}^n \partial_j\big(a_{j,k}(x,\ell_{j,k}(u))\partial_k v\big)+\sum_{j=1}^na_j(x,\ell_j(u))\partial_j v+a_0(x,\ell_0(u))v\,.
\end{align*}
Here, the coefficients satisfy 
\begin{subequations}\label{Z}
\begin{align}\label{z0}
a_{j,k}, a_j, a_0\in C^\infty\big(\bar\Omega\times\R,\R\big)
\end{align} 
and $\ell_{j,k}, \ell_j, \ell_0$ are real-valued smooth (possibly nonlinear) functionals on $L_p(\Omega)$, i.e.
\begin{align}\label{z00}
\ell_{j,k}, \ell_j, \ell_0\in C^\infty\big(L_p(\Omega),\R\big)\,.
\end{align}
 Examples of the latter include linear functionals of the form
$$
\ell(u)=\int_\Omega m(x) u(x)\,\rd x,
$$
with some weight $m\in L_{p'}(\Omega)$. 
We assume in~\eqref{PA2} that either $\delta=0$ (Dirichlet case) or $\delta=1$  (Neumann case). 
Moreover, we assume uniform ellipticity of $\mathcal{A}(x,0)$ in the sense that there exists  a constant~$\underline{a}>0$ such that
\begin{align}\label{z1}
\sum_{j,k=1}^n a_{j,k}(x,\ell_{j,k}(0))\xi^j\xi^k\ge \underline{a}\vert\xi\vert^2\,,\quad (x,\xi)\in\bar\Omega\times\R^n\,,
\end{align}
and impose (for simplicity) the symmetry assumption 
\begin{align}\label{z2}
a_{j,k}(\cdot,\ell_{j,k}(0))=a_{k,j}(\cdot,\ell_{k,j}(0))\,,\quad a_{j}(\cdot,\ell_{j}(0))=0\,,\qquad 1\le j,k\le n\,.
\end{align}
We consider a weight function 
\begin{align}\label{z3}
\ww\in C^1\big([0,T],\R^+)\,,\qquad \ww(0)>0\,.
\end{align}
As for the semilinearity we assume that
\begin{align}\label{z4}
f\in C^1(\R,\R)\,,\qquad f(0)=f'(0)=0\,,
\end{align}
and there is $\ell\ge 1$ such that
\begin{align}\label{z5}
\vert f'(r)-f'(s)\vert\le c\big(\vert r\vert^{\ell-1}+\vert s\vert^{\ell-1}\big)\vert r-s\vert\,,\quad r,s\in\R\,.
\end{align}
\end{subequations}
We can now show that the problem~\eqref{PA} is well-posed under these assumptions. Indeed, let 
$$
E_0:=L_p(\Omega)\,,\qquad E_1:=H_{p,\mathcal{B}}^2(\Omega):=\{u\in H_p^2(\Omega)\,;\, \mathcal{B}u=0 \text{ on } \partial\Omega\}
$$ for some $p\ge 2$ (to be specified later). Setting
$$
A_p(u):=\mathcal{A}(\cdot,u)\vert_{H_{p,\mathcal{B}}^2(\Omega)}\,,\quad u\in L_p(\Omega)\,,
$$ 
it follows from \eqref{z0}-\eqref{z00} that
\begin{align}\label{H1}
A_p\in C^\infty\big(L_p(\Omega),\mathcal{L}(H_p^{2+2\theta}(\Omega),H_p^{2\theta}(\Omega))\big)\,,\quad  2\theta\ge 0\,,
\end{align}
while~\eqref{z1}-\eqref{z2} imply (see \cite{Amann_Teubner}) that 
\begin{align}\label{H2}
A_p(0)\in \mathcal{H}\big(H_{p,\mathcal{B}}^2(\Omega),L_p(\Omega)\big)
\end{align}
with spectrum consisting of countably many real eigenvalues and independent of $p\ge 2$. 
In particular, the Fourier series expansion of
$$
\Phi_{A_2(0)}=\int_0^T\ww(t)e^{tA_2(0)}\,\rd t\in\mathcal{K}(L_2(\Omega))
$$
 with respect to an orthonormal basis in $L_2(\Omega)$ consisting of eigenfunctions of~$A_2(0)$
has coefficients $\int_0^T\ww(t)e^{t\lambda}\,\rd t>0$ with $\lambda\in \sigma(A_2(0))$ (due to~\eqref{z3}) and is thus injective. 
 In view of the property $\Phi_{A_2(0)}\vert_{L_p(\Omega)}=\Phi_{A_p(0)}$ for $p\ge 2$, it thus follows that
\begin{align}\label{H3}
\mathrm{ker}\big(\Phi_{A_p(0)}\big)=\{0\}\,.
\end{align}
Next, using $f'(0)=0$, we deduce exactly as in \cite[Lemma~4.1]{MW_PRSE} from~\eqref{z4} that
\begin{align}\label{H4}
\|f(u)-f(v)\|_{W_p^{2\theta}}\le c\big(\|u\|_{W_p^{2\theta}}^\ell+\|v\|_{W_p^{2\theta}}^\ell\big)\|u-v\|_{W_p^{2\theta}}\,,\qquad u,v\in W_p^{2\theta}(\Omega)\,,
\end{align}
provided that $2\theta>n/p$.
Now,  \cite[Theorem~7.1; Equation (7.5)]{Amann_Teubner} and \cite[Theorem 5.5.1]{Tr78} yield again that the interpolation-extrapolation scale generated by $(L_p(\Omega),\omega-A_p(0))$ (with $\omega>0$ sufficiently large) and  the complex interpolation functor~$[\cdot,\cdot]_\theta$  is given by
\begin{equation}\label{H5}
E_\theta\doteq H_{p,\mathcal{B}}^{2\theta}(\Omega):=\left\{\begin{array}{ll} \{v\in H_{p,\mathcal{B}}^{2\theta}(\Omega) \,:\, \mathcal{B} v=0 \text{ on } 
 \partial\Omega\}\,, &1+1/p<2\theta< 2+\delta+1/p \,,\\[3pt]
	 H_{p,\mathcal{B}}^{2\theta}(\Omega)\,, & 0\le 2\theta<1+1/p\,.\end{array} \right.
\end{equation}
Let $(\ell+1)n<p<\infty$ and choose 
\begin{equation}\label{H6}
 0<\alpha<\alpha_0<\min\left\{\frac{n}{p}, \frac{1}{2}\left(\delta+\frac{1}{p}\right)\right\}\,.
\end{equation}
Thus, we find $\xi$ such that
\begin{equation}\label{H7}
\frac{n}{p}<\xi<\min\left\{\alpha+\frac{1}{\ell+1} ,1\right\}
\end{equation}
 and   $\gamma$ with
\begin{equation}\label{H8}
0=:\beta<\alpha<\alpha_0<\gamma<\xi\,.
\end{equation}
Observe then from~\eqref{H4},~\eqref{H5},~\eqref{H7}, and~\eqref{H8} that
\begin{align}\label{H44}
\|f(u)-f(v)\|_{E_\gamma}\le c\big(\|u\|_{E_\xi}^\ell+\|v\|_{E_\xi}^\ell\big)\|u-v\|_{E_\xi}\,,\quad u,v\in E_\xi\,,
\end{align}
while \eqref{H1}, \eqref{H2}, and~\eqref{H5} give
\begin{align}\label{H11}
A_p\in C^\infty\big(E_0,\mathcal{L}(E_{1+\theta},E_\theta)\big)\,,\quad \theta\in\{0,\alpha,\alpha_0\}\,,\qquad A_p(0)\in \mathcal{H}(E_1,E_0)\,.
\end{align}
Consequently, writing~\eqref{PA} in the form
$$
u'=A_p(u)u+f(u)\,,\quad t\in (0,T)\,,\qquad \int_0^T\ww(t) u(t)\,\rd t=M\,,
$$
and gathering~\eqref{H3},~\eqref{H8},~\eqref{H44}, and~\eqref{H11}, we may invoke Theorem~\ref{MT1} to deduce:

\begin{thm}\label{T:Exem2}
Assume \eqref{Z}. Consider $(\ell+1)n<p<\infty$ and   let  $\alpha$, and $\xi$ be as in~\eqref{H6} respectively~\eqref{H7}.
Then, there is a constant~$m_0>0$ such that for each $M\in H_p^{2+2\alpha}(\Omega)$  with~$\mathcal{B} M=0$ on~$\partial\Omega$ and $\|M\|_{H_p^{2+2\alpha}}\leq m_0$, 
there exists a unique  solution to~\eqref{PA} satisfying
\begin{align*}
u&\in C^{\min\{\alpha,1-\mu(\ell+1)\}}\big([0,T],L_p(\Omega)\big)\cap C\big([0,T],H_p^{2\alpha}(\Omega)\big)\cap C\big((0,T],H_p^2(\Omega)\big)\cap C^{1}\big((0,T],L_p(\Omega)\big)
\end{align*}
such that $\lim_{t\to0^+}t^\mu \|u(t)\|_{H^{2\xi}_p}=0$ for $ \xi-\alpha<\mu<(\ell+1)^{-1}$.
\end{thm}

Although the assumptions and, in particular, the choice of the parameter~$p$ in Theorem~\ref{T:Exem2} are sufficient for our purposes, they are not claimed to be optimal and may admit further refinement.

\section{Quasilinear Equations with Other Nonlocal Conditions}\label{Sec4}

As announced in the introduction one may also consider the evolution equation~\eqref{EEeq} subject to other nonlocal conditions such as~\eqref{other1} and~\eqref{other2}.
These problems are studied below.

\subsection*{Well-Posedness for the Nonlocal Condition~\eqref{other1}}

We consider
\begin{equation}\label{EE2}
u'=A(u)u+f(u)\,,\quad t\in(0,T]\,,\qquad u(0)+\int_0^T  \ww(t)\, u(t)\,\rd t=M\,.
\end{equation}
Formally, a solution to~\eqref{EE2} solves 
\begin{equation*}\label{EE2x}
u(t)=U_{A(u)}(t,0)\big(1+\Phi_{A(u)}\big)^{-1}\big(M+\Psi(u) f(u)\big)+\int_0^t U_{A(u)}(t,s)f(u(s))\,\rd s\,,\quad t\in [0,T]\,,
\end{equation*}
with $\Phi_{A(u)}$ and  $\Psi(u)$ defined in \eqref{Phi} and \eqref{Psi}, respectively.
This problem is substantially easier than~\eqref{EE} since, under suitable additional assumptions, the operator~$(1+\Phi_{A(u)})^{-1}$ is bounded on~$E_0$ (and on any interpolation space) and is thus of order zero. Therefore, we do not need to shift to extrapolation spaces of higher regularity and may require correspondingly weaker conditions. 

More precisely, we assume that
\begin{subequations}\label{AAxx}
\begin{equation}\label{a2x}
\text{$E_1$ is compactly and densely embedded in $E_0$}\,.
\end{equation}
Given $\theta\in (0,1)$, we fix an exact admissible interpolation functor $(\cdot,\cdot)_\theta$  of exponent $\theta$ and introduce again the Banach space $E_\theta:=(E_0,E_1)_\theta$.
Let
\begin{equation}\label{a3bx}
0\le \beta<\alpha< 1
\end{equation}
 and    assume that there is $r_0>0$ such that
\begin{equation}\label{a5xx}
A\in C^{1-}\big(\mathbb{B}_{E_\beta}(0,r_0), \mathcal{L}(E_1,E_0)\big)\,,\qquad A(0)\in \mathcal{H}(E_1,E_0)\,.
\end{equation}
As for the weight function we merely require that
\begin{equation}\label{a1x}
\ww\in C([0,T])\,,
\end{equation}
\end{subequations}
while for the semilinear part we fix
\begin{subequations}\label{AAf}
\begin{equation}\label{A1Yf}
 \ell>0\,,\qquad 0<\gamma\le 1\,,\qquad 0<\xi<\min\left\{ \alpha+\frac{1-(\alpha-\gamma)_+}{\ell+1},1\right\}\,,
\end{equation}
and assume that $f:E_\xi\cap \mathbb{B}_{E_\beta}(0,r_0)\to E_\gamma$ 
satisfies for $v,w\in E_\xi\cap \mathbb{B}_{E_\beta}(0,r_0)$
\begin{equation}\label{F2f}
\|f(v)-f(w)\|_{E_\gamma}\le c(r_0) \left(\left[\|v\|_{E_\xi}^\ell+\|w\|_{E_\xi}^\ell\right] \|v-w\|_{E_\xi}+ \left[\|v\|_{E_\xi}^{\ell+1}+\|w\|_{E_\xi}^{\ell+1}\right] \|v-w\|_{E_\beta}\right)
\end{equation}
for some constant $c(r_0)>0$, and that
\begin{equation}\label{F2xf}
\qquad f(0)=0\,.
\end{equation}
\end{subequations}
Moreover, the operator $\Phi_{A(0)}\in \mathcal{K}(E_0)$  (see the arguments leading to \eqref{p2})  is assumed to satisfy
\begin{equation}\label{o1}
-1\notin \sigma_p(\Phi_{A(0)})\,.
\end{equation}
Under these conditions, we can show that \eqref{EE2} is well-posed:

\begin{thm}\label{MT3}
Assume~\eqref{AAxx}, \eqref{AAf}, and \eqref{o1}.
 Then, there exists $m_0>0$ such that
\eqref{EE2}
 has for each $M\in E_{\alpha}$ with $\|M\|_{E_{\alpha}}\le m_0$ a unique solution
$$
u\in C^{\min\{\alpha-\beta, 1-\mu(\ell+1)\}}\big([0,T],E_\beta\big)\cap C\big([0,T],E_\alpha\big)\cap C\big((0,T],E_1\big)\cap C^{1}\big((0,T],E_0\big)
$$
such that  $\lim_{t\to 0^+}t^\mu\|u(t)\|_{E_\xi}=0$ for   $(\xi-\alpha)_+<\mu< (\ell+1)^{-1}$.
\end{thm}

\begin{proof}
The proof  relies on Banach's fixed point theorem.
To start, we  choose  $\mu$ and~$\rho$ such that
 \begin{equation*}
  (\xi-\alpha)_+<\mu<\frac{1-(\alpha-\gamma)_+}{\ell+1}\qquad\text{and}\qquad
 0<\rho<\min\{\alpha-\beta, 1-\mu(\ell+1)\}\,.
\end{equation*} 
For $L\in (0,r_0)$, we then set
\begin{align*}
Y_L := \left\{ u \in C([0,T],E_\beta) \cap C_\mu((0,T],E_\xi) \,:\, 
\text{%
\parbox[c]{0.425\linewidth}{%
$\|u(t)\|_{C([0,T],E_\beta)} + \|u\|_{C_\mu((0,T],E_\xi)} \le L$,\\[1mm]
$\|u(t)-u(s)\|_{E_\beta} \le (t-s)^\rho\,,\, 0 \le s \le t \le T$
}%
} \right\}\,.
\end{align*}
This closed subspace of $ C\big([0,T],E_\beta\big)\cap  C_\mu\big((0,T],E_\xi\big)$ is equipped with the distance function
$$
d_{Y_L}(u,\bar u):=\|u-\bar u\|_{C([0,T],E_\beta)} +\|u-\bar u\|_{C_\mu((0,T],E_\xi)} \,,\quad u, \bar u\in Y_L\,,
$$
and is thus a complete metric space.
Making $r_0>0$ smaller if necessary, we first note from \eqref{a5xx} that $A(u)\in C^{\rho}\big([0,T],\mathcal{H}(E_1,E_0)\big)$ for any $u\in  Y_L$, hence 
\begin{equation}\label{o3}
\Phi_{A(u)}\in \mathcal{L}(E_0,E_\alpha)\cap\mathcal{L}(E_\alpha,E_{1})\cap  \mathcal{K}(E_\alpha)
\end{equation}
 is well-defined due to~\eqref{a1x}  \cite[II.Lemma~5.1.3]{LQPP}. 
 Then \cite[II.Lemma~5.1.4]{LQPP} implies that we may choose $L\in (0,r_0)$ small enough such that
\begin{equation}\label{o4}
\| \Phi_{A(u)}-\Phi_{A(\bar u)}\|_{\mathcal{L}(E_\alpha)}\le c_0 \|u-\bar u\|_{C([0,T], E_\beta)}\,,\quad u, \bar u\in Y_L\,.
\end{equation}
Now, \eqref{o1} and \eqref{o3} ensure that $1+\Phi_{A(0)}\in \mathcal{L}is(E_\alpha)$. Since $\mathcal{L}is(E_\alpha)$ is open in~$\mathcal{L}(E_\alpha)$, it follows from~\eqref{o4} (making $L>0$ smaller, if necessary) that
$1+\Phi_{A(u)}\in \mathcal{L}is(E_\alpha)$ for $u\in Y_L$ and
\begin{equation}\label{o5}
\| (1+\Phi_{A(u)})^{-1}-(1+\Phi_{A(\bar u)})^{-1}\|_{\mathcal{L}(E_\alpha)}\le c_0 \|u-\bar u\|_{C([0,T], E_\beta)}\,,\quad u, \bar u\in Y_L\,.
\end{equation}
Analogously to Lemma~\ref{psi}, we infer for $\nu:=\mu(\ell+1)\in (0,1)$  that there is a constant~$c=c(L)>0$ such that
\begin{align}\label{o5a}
\|\Psi(u)\|_{\mathcal{L}(C_\nu((0,T],E_\gamma),E_{\alpha})}\le c
\end{align}
and
\begin{align}\label{o5aa}
\|\Psi(u)-\Psi(\bar u)\|_{\mathcal{L}(C_\nu((0,T],E_\gamma),E_{\alpha})}\le c\|u-\bar u\|_{C([0,T],E_\beta)}
\end{align}
Defining
$$
\Sigma(u):=(1+\Phi_{A(u)})^{-1}\big(M+\Psi(u) f(u))\,,\quad u\in Y_L\,,
$$
and recalling~\eqref{AAf} (see also~\eqref{n1}-\eqref{n2}), it thus follows from~\eqref{o5},~\eqref{o5a}, and~\eqref{o5aa} that $\Sigma$ satisfies
\begin{subequations}\label{SigmaA}
\begin{equation}
\|\Sigma(u)-\Sigma(\bar u)\|_{E_\alpha}\le c\big(\|M\|_{E_\alpha}+L^\ell+L^{\ell+1}\big) d_{Y_L}(u,\bar u)\,,\quad u,\bar u\in Y_L\,,
\end{equation}
and
\begin{equation}
\|\Sigma(u)\|_{E_\alpha}\le c\big(\|M\|_{E_\alpha}+L^{\ell+1}\big) \,,\quad u\in Y_L\,.
\end{equation}
\end{subequations}
It is now a standard argument to show that the mapping 
$$
u\mapsto U_{A(u)}(t,0)\Sigma(u)+\int_0^t U_{A(u)}(t,s)f(u(s))\,\rd s
$$ 
is a contraction on $Y_L$ provided that $\|M\|_{E_\alpha}$ and $L$ are sufficiently small. Hence, using the regularity result of \cite[II.Theorem~1.2.2]{LQPP} we conclude that~\eqref{EE2x} has a unique solution as stated in Theorem~\ref{MT3}.
\end{proof}

\subsection*{Well-Posedness for the Nonlocal Condition~\eqref{other2}}

Finally, we consider the problem
\begin{equation}\label{EE3}
u'=A(u)u+f(u)\,,\quad t\in(0,T]\,,\qquad u(0)-  \aw u(T)=M\,,
\end{equation}
with $\aw\in\R$. In this case, solutions obey the fixed point equation
\begin{equation}\label{EE3x}
u(t)=U_{A(u)}(t,0)\big(1- \aw U_{A(u)}(T,0)\big)^{-1}\big(M+N_u(T)f(u)\big)+N_u(t)f(u)\,,\quad t\in [0,T]\,,
\end{equation}
with 
$$
N_u(t)g:=\int_0^tU_{A(u)}(t,s)g(s)\,\rd s\,,\quad t\in [0,T]\,.
$$
Also this problem is much simpler than~\eqref{EE} and weaker assumptions are sufficient.
That is, in this case we assume~\eqref{AAxx}-\eqref{AAf} together with
\begin{equation}\label{o1y}
1\notin  \aw e^{T\sigma(A(0))}\,.
\end{equation}

Then, the well-posedness of \eqref{EE3} can also be guaranteed:

\begin{thm}\label{MT4}
Assume~\eqref{AAxx}-\eqref{AAf} and~\eqref{o1y}.
 Then, there exists $m_0>0$ such that~\eqref{EE3}
 has for each $M\in E_{\alpha}$ with $\|M\|_{E_{\alpha}}\le m_0$ a unique solution
$$
u\in C^{\min\{\alpha-\beta, 1-\mu(\ell+1)\}}\big([0,T],E_\beta\big)\cap C\big([0,T],E_\alpha\big)\cap C\big((0,T],E_1\big)\cap C^{1}\big((0,T],E_0\big)\,.
$$
such that  $\lim_{t\to 0^+}t^\mu\|u(t)\|_{E_\xi}=0$ for   $(\xi-\alpha)_+<\mu< (\ell+1)^{-1}$.
\end{thm}

\begin{proof}
We can fix  
\begin{equation*}
  (\xi-\alpha)_+<\mu<\frac{1-(\alpha-\gamma)_+}{\ell+1}\qquad\text{and}\qquad
 0<\rho<\min\{\alpha-\beta, 1-\mu(\ell+1)\}\,.
\end{equation*} 
The proof now is the same as for Theorem~\ref{MT3} noticing that \eqref{o1y} implies that 
\begin{equation}\label{t3}
1- \aw U_{A(0)}(T,0)=1- \aw e^{TA(0)}\in\mathcal{L}is(E_\alpha)
\end{equation}
and hence $1- \aw U_{A(u)}(T,0)\in\mathcal{L}is(E_\alpha)$
for $u\in Y_L$ with $L>0$ small enough and that
$$
\Sigma(u):=\big(1- \aw U_{A(u)}(T,0)\big)^{-1}\big(M+N_u(T)f(u)\big)
$$
satisfies~\eqref{SigmaA} since, for $\nu:=\mu(\ell+1)\in (0,1)$, $\theta\in\{\alpha,\xi\}$, and $t>0$,
$$
\|N_u(t)g\|_{E_\theta}\le c(1+t^{1+\gamma_0-\theta-\nu})\|g\|_{C_\nu((0,t],E_\gamma)}\,,
$$
with $\gamma_0\in(0,\gamma)$ chosen such that  $\mu(\ell+1)<1+\gamma_0+\alpha$, and
$$
 \|(N_u(t)-N_{\bar u}(t))g\|_{E_\theta}\le ct^{1+\gamma-\theta-\nu}\|u-\bar u\|_{C([0,T],E_\beta)}\|g\|_{C_\nu((0,t],E_\gamma)}\,,
$$
where $c=c(L)>0$ is a suitable constant. This yields Theorem~\ref{MT4}.
\end{proof}

\bibliographystyle{siam}
\bibliography{LiteratureQuasiRecov}
\end{document}